\documentclass[12pt]{amsart}
\usepackage{amsmath}
\usepackage{amsfonts}
\usepackage{amssymb}
\usepackage[all,cmtip]{xy}           
\usepackage{bbm}
\usepackage{bbding}
\usepackage{txfonts}
\usepackage[shortlabels]{enumitem}
\usepackage{ifpdf}
\ifpdf
\usepackage[colorlinks,final,backref=page,hyperindex]{hyperref}
\else
\usepackage[colorlinks,final,backref=page,hyperindex,hypertex]{hyperref}
\fi
\usepackage{mathrsfs}
\usepackage{tikz-cd}
\usepackage[active]{srcltx}
\usepackage{tikz}
\usepackage{amscd}
\usepackage{bm}
\usepackage{multirow}
\usetikzlibrary{decorations.pathreplacing}
\usetikzlibrary{calc}
\usetikzlibrary{arrows,shapes,chains}

\newtheorem{thm}{Theorem}[section]
\newtheorem{prop}[thm]{Proposition}
\newtheorem{lem}[thm]{Lemma}
\newtheorem{cor}[thm]{Corollary}
\newtheorem{prop-def}{thm}[section]

\theoremstyle{definition}
\newtheorem{defn}[thm]{Definition}

\newtheorem{remark}[thm]{Remark}
\newtheorem{exam}[thm]{Example}

\newcommand{\nc}{\newcommand}

 \nc{\mbibitem}[1]{\bibitem{#1}} 
 \nc{\mrm}[1]{{\rm #1}}

\nc{\ac}{\mathrm{\textup{!`}}}
\nc{\bs}{\bar{S}}
\nc{\pl}{\cdot}
 \nc{\la}{\longrightarrow}
\nc{\ot}{\otimes}
 \nc{\rar}{\rightarrow}
 \nc{\btr}{\blacktriangleright}
 \nc{\btl}{\blacktriangleleft}

\nc{\PLA}{{\mathrm{RBS}}}
\nc{\bfk}{{\bf k}}
\nc{\Alg}{\mathrm{Alg}}
\nc{\C}{{\mathrm{C}}}
\nc{\RBA}{{\mathrm{RBA}_\lambda}}
\nc{\RBO}{\mathsf{RBSO}}
\nc{\End}{\mrm{End}}
\nc{\Ext}{\mrm{Ext}}
\nc{\Fil}{\mrm{Fil}}
\nc{\Fr}{\mrm{Fr}}
\nc{\Frob}{\mrm{Frob}}
\nc{\Gal}{\mrm{Gal}}
\nc{\GL}{\mrm{GL}}
\nc{\Hom}{\mrm{Hom}}
\nc{\Hoch}{\mrm{Hoch}}
\nc{\hsr}{\mrm{H}}
\nc{\hpol}{\mrm{HP}}
\nc{\im}{\mrm{im}}
\nc{\Id}{\mrm{Id}}
\nc{\id}{\mrm{Id}}
\nc{\h}{\mrm{H}}
\nc{\Alt}{\mrm{Alt}}
\nc{\Irr}{\mrm{Irr}}
\nc{\incl}{\mrm{incl}}
\nc{\length}{\mrm{length}}
\nc{\NLSW}{\mrm{NLSW}}
\nc{\Lie}{\mrm{Lie}}
\nc{\mchar}{\rm char}
\nc{\mpart}{\mrm{part}}
\nc{\ql}{{\QQ_\ell}}
\nc{\qp}{{\QQ_p}}
\nc{\rank}{\mrm{rank}}
\nc{\rcot}{\mrm{cot}}
\nc{\rdef}{\mrm{def}}
\nc{\rdiv}{{\rm div}}
\nc{\rmH}{ {\mathrm{H}}}
\nc{\rtf}{{\rm tf}}
\nc{\rtor}{{\rm tor}}
\nc{\res}{\mrm{res}}
\nc{\Sh}{{\mathrm{Sh}}}
\nc{\sh}{\mathrm{\overline{Sh}}}
\nc{\SL}{\mrm{SL}}
\nc{\Spec}{\mrm{Spec}}
\nc{\sgn}{{\mathrm{sgn}}}
\nc{\tor}{\mrm{tor}}
\nc{\Tr}{\mrm{Tr}}
\nc{\tr}{\mrm{tr}}
\nc{\wt}{\mrm{wt}}
\nc{\op}{\mrm{op}}
\nc{\s}{\mrm{S}}
\nc{\ra}{\rightarrow}
\nc{\rH}{\mathrm{H}}
\nc{\RBL}{\mathfrak{RBS}}

\nc{\RBLinfty}{{\mathfrak{RBS}_\infty}}
\nc{\Int}{\mathbf{Int}}
\nc{\Lea}{\mathbf{Leaves}}
\nc{\Pare}{\mathbf{Parent}}
\def\lbb{\{\kern -.19em\{}
\def\rbb{\}\kern -.19em\}}

\nc{\BA}{{\mathbb A}}   \nc{\CC}{{\mathbb C}}
\nc{\DD}{{\mathbb D}}   \nc{\EE}{{\mathbb E}}
\nc{\FF}{{\mathbb F}}   \nc{\GG}{{\mathbb G}}
\nc{\HH}{ \mathrm{HH}}   \nc{\LL}{{\mathbb L}}
\nc{\NN}{{\mathbb N}}   \nc{\PP}{{\mathbb P}}
\nc{\QQ}{{\mathbb Q}}   \nc{\RR}{{\mathbb R}}
\nc{\TT}{{\mathbb T}}   \nc{\VV}{{\mathbb V}}
\nc{\ZZ}{{\mathbb Z}}   \nc{\TP}{\widetilde{P}}
\nc{\m}{{\mathbbm m}}

\nc{\cala}{{\mathcal A}}    \nc{\calc}{{\mathcal C}}
\nc{\cald}{\mathcal{D}}     \nc{\cale}{{\mathcal E}}
\nc{\calf}{{\mathcal F}}    \nc{\calg}{{\mathcal G}}
\nc{\calh}{{\mathcal H}}    \nc{\cali}{{\mathcal I}}
\nc{\call}{{\mathcal L}}    \nc{\calm}{{\mathcal M}}
\nc{\caln}{{\mathcal N}}    \nc{\calo}{{\mathcal O}}
\nc{\calp}{{\mathcal P}}    \nc{\calr}{{\mathcal R}}
\nc{\cals}{{\mathcal S}}    \nc{\calt}{{\Omega}}
\nc{\calv}{{\mathcal V}}    \nc{\calw}{{\mathcal W}}
\nc{\calx}{{\mathcal X}}

\nc{\fraka}{{\mathfrak a}}
\nc{\frakb}{\mathfrak{b}}
\nc{\frakg}{{\frak g}}
\nc{\frakl}{{\frak l}}
\nc{\fraks}{{\frak s}}
\nc{\frakt}{{\mathfrak{T}}}
\nc{\frakm}{{\frak m}}
\nc{\frakM}{{\frak M}}
\nc{\frakp}{{\frak p}}
\nc{\frakW}{{\frak W}}
\nc{\frakX}{{\frak X}}
\nc{\frakS}{{\frak S}}
\nc{\frakA}{{\frak A}}
\nc{\frakC}{{\frak{C}}}
\nc{\frakx}{{\frakx}}
\nc{\red}{\color{red}}

\nc{\kai}[1]{\textcolor{blue}{\underline{Kai:}#1 }}

\begin{document}

\title[ From homotopy RB-algebras to pre-CY algebras and double Poisson algebras]{From homotopy Rota-Baxter algebras to Pre-Calabi-Yau and homotopy double Poisson algebras}

\author{Yufei Qin }
\address{Department of Mathematics and Data Science\\ Vrije Universiteit Brussel\\ Pleinlaan 2, 1050 Brussels\\ Belgium  School of Mathematical Sciences\\ Key Laboratory of Mathematics and Engineering Applications (Ministry of Education)\\ Shanghai Key laboratory of PMMP\\	East China Normal University\\	Shanghai 200241,	China}
 \email{Yufei.Qin@vub.be}

\author{Kai Wang}
\address{ School of Mathematical Sciences\\ University of Science and Technology of China\\ Hefei, Anhui Provience 230026, China}
  
   \email{wangkai17@ustc.edu.cn }

\date{\today}

\begin{abstract}

In this paper, we investigate pre-Calabi-Yau algebras and homotopy double Poisson algebras arising from homotopy Rota-Baxter structures. We introduce the notion of cyclic homotopy Rota-Baxter algebras, a class of homotopy Rota-Baxter algebras endowed with additional cyclic symmetry, and present a construction of such structures via a process called cyclic completion. We further introduce the concept of interactive pairs, consisting of two differential graded  algebras—designated as the acting algebra and the base algebra—interacting through compatible module structures. We prove that if the acting algebra carries a suitable cyclic homotopy Rota-Baxter structure, then the base algebra inherits a natural pre-Calabi-Yau structure. Using the correspondence established by Fernández and Herscovich between pre-Calabi-Yau algebras and homotopy double Poisson algebras, we describe the resulting homotopy Poisson structure on the base algebra in terms of homotopy Rota-Baxter algebra structure. In particular, we show that a module over an ultracyclic (resp. cyclic) homotopy Rota-Baxter algebra admits a (resp. cyclic) homotopy double Lie algebra structure.

\end{abstract}

\subjclass[2020]{
	17B38  	
16E45 
17B63  	
18N70  	
14A22 
}

\keywords{  Rota-Baxter algebra; pre-Calabi-Yau algebra; double Poisson algebra. }

\maketitle

\tableofcontents

\allowdisplaybreaks

\section*{Introduction}

The concept of double Poisson algebras was introduced by Van den Bergh, who used it to develop a foundational framework for noncommutative Poisson geometry \cite{VdB08}. He demonstrated that the representation scheme of such an algebra naturally inherits a classical Poisson structure. From this perspective, double Poisson algebras provide a natural and robust setting for formulating noncommutative Poisson geometry, aligning with the Kontsevich-Rosenberg principle.

In a parallel development, Kontsevich, Takeda, and Vlassopoulos \cite{KTV25} introduced the notion of pre-Calabi-Yau algebras (or more generally, pre-Calabi-Yau categories), which can also be considered as a framework for noncommutative Poisson geometry.  Iyudu, Kontsevich, and Vlassopoulos \cite{IKV21} showed that the representation spaces of a certain class of pre-Calabi-Yau algebras naturally carry classical Poisson structures. More generally, Yeung \cite{Yeun22} demonstrated that the derived moduli stack of a pre-Calabi-Yau algebra admits a shifted Poisson structure. Thus, pre-Calabi-Yau algebras offer an equally compelling and versatile framework for developing noncommutative Poisson geometry.

Both double Poisson algebras and pre-Calabi-Yau algebras provide frameworks for noncommutative Poisson geometry, suggesting an intrinsic connection between the two concepts. Iyudu, Kontsevich, and Vlassopoulos established a bijection between double Poisson algebras and  a special type of  pre-Calabi-Yau algebras \cite{IKV21},  this correspondence was later given a conceptual interpretation via higher cyclic Hochschild cohomology in \cite{IK}. Subsequently,  Fern\'andez and Herscovich extended the bijection to differential graded (dg) setting and further  to homotopy double Poisson algebras \cite{FH21}. Specially, they proved that there is a bijection between a particular class of pre-Calabi-Yau algebra (called good manageable special pre-Calabi-Yau algebras) and homotopy double Poisson algebras. Later, they also proved that double quasi-Poisson algebra are also pre-Calabi-Yau algebras \cite{FH22}. Recently, using the methods of properad theory, Leray and Vallette proved the equivalence between curved pre-Calabi-Yau algebras and curved double Poisson algebras by showing that the differential graded Lie algebras governing their deformation theories are quasi-isomorphic \cite{LV}.

The concept of homotopy double Poisson algebra, was introduced by Schedler \cite{Sched09}. In the same work,  he also formulated the associative Yang-Baxter-infinity equation and studied the relationship between homotopy double Poisson algebras and associative Yang-Baxter-infinity equation. In particular, he proved that there is a bijection between the skew-symmetric solutions of associative Yang-Baxter-infinity equation and homotopy double Lie algebras--that is, homotopy double Poisson algebras with the multiplication forgotten.  Leray introduced the concept of protoperads (an analogue of operads) and showed that the protoperad governing double Poisson algebras is Koszul \cite{Leray22, Leray20}, leading to a natural construction of the minimal model of protoperad governing double Poisson algebras, which  generalizes Schedler's homotopy double Poisson algebras.

In this paper, we focus on constructing pre-Calabi-Yau algebras and homotopy double Poisson algebras from representations of homotopy Rota-Baxter algebras.

Rota-Baxter algebras, originally introduced by G. Baxter in the context of probability theory \cite{Bax60}, were later developed by Rota \cite{Rot69}, Cartier \cite{Car72}, and others. This led to the now widely used term “Rota-Baxter algebras.” The theory saw a revival through the work of Guo and collaborators \cite{Agu01, GK00a, GK00b}. Today, Rota-Baxter algebras are connected to numerous areas of mathematics, including combinatorics \cite{Rot69}, renormalization in quantum field theory \cite{CK00}, multiple zeta values in number theory \cite{GZ08}, operad theory \cite{BBGN13}, Hopf algebras \cite{CK00}, and Yang-Baxter equations \cite{Bai07}. For an accessible overview, see Guo’s introduction \cite{Guo09}; for a comprehensive treatment, refer to his monograph \cite{Guo12}. Das and Misha \cite{DM20} studied deformations of relative Rota-Baxter associative algebras and introducing the notion of homotopy relative Rota-Baxter algebras. Building on operadic methods, Wang and Zhou \cite{WZ24} constructed the minimal model of the operad governing Rota-Baxter associative algebras of arbitrary weight. From this, they derived the corresponding $L_\infty$-algebra governing deformations and introduced the concept of homotopy Rota-Baxter algebras of arbitrary weight. The deformation and homotopy theories of Rota-Baxter structures on Lie algebras have also been studied by Tang, Bai, Guo, and Sheng \cite{TBGS19}, as well as by Lazarev, Sheng, and Tang \cite{LST21}.

  In 1983, Semenov-Tian-Shansky \cite{STS83} showed that a solution to the classical Yang-Baxter equation in a Lie algebra induces a Rota-Baxter operator on that Lie algebra. Later, Kupershmidt \cite{Kup99} demonstrated that a skew-symmetric solution yields a relative Rota-Baxter operator. On the associative side, Aguiar introduced the associative Yang-Baxter equation \cite{Agu00b} and showed that its solutions naturally endow associative algebras with Rota-Baxter operators \cite{Agu00a}. Subsequently, Gubarev \cite{Gub21}, and independently Zhang, Gao, and Zheng \cite{ZGZ18p}, established a one-to-one correspondence between solutions of the associative Yang-Baxter equation and Rota-Baxter algebra structures on matrix algebras. Building on this, Goncharov and Kolesnikov \cite{GK18} introduced the notion of a skew-symmetric Rota-Baxter operator, and proved that such operators on $M_n(\mathbf{k})$ are equivalent to double Lie algebra structures on the n-dimensional vector space over $\mathbf{k}$.
  
  These results reveal a deep interplay among Rota-Baxter algebras, double Poisson algebras, and pre-Calabi-Yau algebras. In this paper, we explore these connections within a more general homotopical framework. We introduce the notions of cyclic and ultracyclic homotopy relative Rota-Baxter algebras, where the homotopy Rota-Baxter structures satisfy certain cyclic invariance conditions. These notions generalize the skew-symmetric Rota-Baxter operators studied by Goncharov and Kolesnikov in \cite{GK18}. To investigate the pre-Calabi-Yau and double Poisson structures arising from homotopy Rota-Baxter algebras, we also define the concept of interactive pairs: pairs of differential graded algebras $(A, B)$, referred to as the acting algebra $A$ and the base algebra $B$, which act on each other in a compatible manner. In particular, we consider interactive pairs in which the acting algebra $ A $ is equipped with a suitable homotopy relative Rota–Baxter structure. Such pairs will be called homotopy Rota–Baxter interactive pairs. We then show that  the base algebra $B$ of a homotopy Rota-Baxter naturally acquires a pre-Calabi-Yau algebra structure. More precisely, we establish the following result (see Theorem~\ref{Thm: Pre-Calabi-Yau structure on homotopy RB dg algebras}):
\begin{thm}\label{Theorem:1}
	Let $(A, B)$ be a homotopy Rota-Baxter interactive pair, where the acting algebra $A$ and the base algebra  $B$   are locally finite-dimensional. Let $\{T_n:(A^\vee)^{\otimes n}\rightarrow A\}_{n\geqslant 1}$ be the homotopy relative Rota-Baxter operator.
	\begin{itemize}
		\item[\rm(i)] If each $T_n$ is cyclic, then $B$ admits a good manageable pre-Calabi-Yau algebra structure.
		\item[\rm(ii)] If each $T_n$ is ultracyclic, then $B$ admits a good manageable special pre-Calabi-Yau algebra structure.
	\end{itemize}
\end{thm}
Then, using the correspondence between pre-Calabi-Yau algebras and homotopy double Poisson algebras established by Fernández and Herscovich in \cite{FH21}, we describe the induced homotopy double Poisson structures on base algebras in terms of homotopy Rota-Baxter algebra structures. This description is given in an explicit and streamlined form (see Theorem~\ref{Thm:From RB infinity alegbra to double Poisson infinity}):
\begin{thm} \label{Theorem: 2}
	Let $(A, B)$ be a homotopy Rota-Baxter interactive pair, where the acting algebra $A$ is finite-dimensional and the base algebra $B$ is locally finite-dimensional. Let $\{T_n : (A^\vee)^{\otimes n} \to A\}_{n \geqslant 1}$ be a relative differential graded homotopy Rota-Baxter operator on $A$.
	
	Define a sequence of maps $\{\lbb -, \ldots, - \rbb_n : B^{\otimes n} \to B^{\otimes n} \}_{n \geqslant 1}$ by setting $\lbb - \rbb_1 = d_B$, and for all $n \geqslant 1$,
	\begin{equation*}
		\lbb -, \ldots, - \rbb_{n+1} := \Psi^n(\id_{A^{\otimes n}}),
	\end{equation*}
	where the map $\Psi^n$ is the composition:
	\[
	\Psi^n : \End(A^{\otimes n}) \cong A^{\otimes n} \otimes (A^\vee)^{\otimes n} 
	\xrightarrow{\id^{\otimes n} \otimes T_n} A^{\otimes (n+1)} 
	\xrightarrow{\Phi^{\otimes (n+1)}} \End(B)^{\otimes (n+1)} 
	\to \End(B^{\otimes (n+1)}),
	\]
	and $\Phi : A \to \End(B)$ denotes the left $A$-action on $B$, i.e., $\Phi(a)(b) := a \rhd b$.
	
	Then, 
	\begin{itemize}
		\item[\rm (i)]If each $T_n$ is  cyclic, the collection $\{\lbb -, \ldots, - \rbb_n\}_{n \geqslant 1}$ defines a cyclic homotopy double Poisson algebra structure on $B$.
		\item[\rm (ii)] If each $T_n$ is ultracyclic, the collection $\{\lbb -, \ldots, - \rbb_n\}_{n \geqslant 1}$ defines a homotopy double Poisson algebra structure on $B$.
		\end{itemize}
\end{thm}

%
%
%
%

\smallskip

The paper is organized as follows:
 
 In Section~\ref{Section: Rota-Baxter algebras and double Poisson algebras}, we recall the definitions of Rota-Baxter algebras and double Lie algebras, along with their known connections.

In Section~\ref{ Cyclic homotopy relative Rota-Baxter algebras}, we begin by reviewing cyclic $A_\infty$-algebras and pre-Calabi-Yau structures. Building on this framework, we introduce cyclic Rota-Baxter algebras, as well as cyclic and ultracyclic homotopy relative Rota-Baxter algebras. We also present a cyclic completion construction for homotopy Rota-Baxter algebras.
 
 In Section~\ref{Section: Pre-Calabi-Yau structures arising from cyclic homotopy Rota-Baxter algebras}, we introduce the notion of interactive pairs. We study homotopy Rota-Baxter structures on the acting algebra of such pairs under certain compatibility conditions, leading to the construction of pre-Calabi-Yau structures on the base algebra. This leads to the proof of Theorem~\ref{Theorem:1} (see Theorem~\ref{Thm: Pre-Calabi-Yau structure on homotopy RB dg algebras}). In particular, we prove that a dg module over a dg algebra equipped with a cyclic homotopy relative Rota–Baxter structure naturally carries a pre-Calabi–Yau algebra structure.
 
 In Section~\ref{Section:Homotopy   Rota-Baxter algebras and double  Poisson structures}, we recall the definitions of homotopy double Lie algebras and homotopy double Poisson algebras. We generalize the correspondence between pre-Calabi-Yau algebras and homotopy double Poisson algebras established by Fernández and Herscovich. Using the constructions from Section~\ref{Section: Pre-Calabi-Yau structures arising from cyclic homotopy Rota-Baxter algebras}, we prove Theorem~\ref{Theorem: 2} (see Theorem~\ref{Thm:From RB infinity alegbra to double Poisson infinity}). As a special case, we show that a dg module over an ultracyclic homotopy relative Rota–Baxter algebra naturally inherits a homotopy double Lie algebra structure. Moreover, we prove that the symmetric algebra of a homotopy double Lie algebra naturally carries a homotopy Poisson algebra structure. This yields a method for constructing homotopy Poisson structures from representations of dg homotopy Rota–Baxter algebras. As an application, we establish an equivalence between skew-symmetric solutions of the associative Yang–Baxter-infinity equations, ultracyclic homotopy Rota–Baxter algebra structures, a certain class of pre-Calabi–Yau algebras, and homotopy double Lie algebras, thus extending the results of Goncharov and Kolesnikov to the homotopical realm.

\newpage

\section{Preliminaries}\

\subsection{Notations} \ 

Let $\bf{k}$ be a field of characteristic $0$.
A (homologically) { graded space} is a $\mathbb{Z}$-indexed family of $\bf{k}$-vector spaces \( V = \{V_n\}_{n \in \mathbb{Z}} \). Elements of \( \bigcup_{n \in \mathbb{Z}} V_n \) are called homogeneous and have a degree \( |v| = n \) if \( v \in V_n \).

Given two graded spaces \( V \) and \( W \), a graded map of degree \( r \) is a linear map \( f: V \to W \) such that \( f(V_n) \subseteq W_{n+r} \) for all \( n \), and we denote the degree of \( f \) by \( |f| = r \). Define 
\[
\mathrm{Hom}(V, W)_r = \prod_{p \in \mathbb{Z}} \Hom_\mathbf{k}(V_p, W_{p+r})
\] 
as the space of graded maps of degree \( r \). The graded space \( \mathrm{Hom}(V, W) \) is then given by \( \{\mathrm{Hom}(V, W)_r\}_{r \in \mathbb{Z}} \).

The tensor product \( V \otimes W \) of two graded spaces \( V \) and \( W \) is defined by  
\[
(V \otimes W)_n = \bigoplus_{p+q=n} V_p \otimes W_q.
\]
We adopt Sweedler’s notation for elements in tensor products of graded spaces. Let \( V^1 \otimes \cdots \otimes V^n \) be the tensor product of graded spaces \( V^1, \cdots, V^n \). An element \( r \) in this tensor product can be expressed as 
\[
r = \sum_{i_1, \cdots, i_n} r^{[1]}_{i_1} \otimes \cdots \otimes r^{[n]}_{i_n},
\] 
where \( r^{[k]}_{i_k} \in V^k \). For simplicity, we omit the subscripts \( i_k \) and write:
\[
r = \sum r^{[1]} \otimes \cdots \otimes r^{[n]}.
\]
If \( V \) is a finite-dimensional graded space, there is an isomorphism of graded spaces:
\[
\mathrm{Hom}(V, W) \cong W \otimes V^\vee.
\]
Moreover, if both \( V \) and \( W \) are finite-dimensional graded spaces, we have the isomorphism:
\[
\End(V \otimes W) \cong V \otimes W \otimes W^\vee \otimes V^\vee.
\]

The suspension of a graded space \( V \) is the graded space \( sV \), defined by \( (sV)_n = V_{n-1} \) for all \( n \in \mathbb{Z} \). For any \( v \in V_{n-1} \), we denote the corresponding element in \( (sV)_n \) by \( sv \). The map \( s: V \to sV \), defined by \( v \mapsto sv \), is a graded map of degree \( 1 \).

Similarly, the desuspension of \( V \), denoted \( s^{-1}V \), is defined by \( (s^{-1}V)_n = V_{n+1} \). For \( v \in V_{n+1} \), the corresponding element in \( (s^{-1}V)_n \) is written as \( s^{-1}v \). The map \( s^{-1}: V \to s^{-1}V \), given by \( v \mapsto s^{-1}v \), is a graded map of degree \( -1 \).

To determine signs in expressions involving graded objects, we use the Koszul sign rule, which states that exchanging the positions of two graded elements introduces a factor of \( (-1)^{|a||b|} \), where \( |a| \) and \( |b| \) are their respective degrees.

Let \( \mathfrak{S}_n \) denote the symmetric group on \( n \) elements, and let \( V \) be a graded space. The left action of \( \mathfrak{S}_n \) on \( V^{\otimes n} \) is defined as follows: for \( \sigma \in \mathfrak{S}_n \) and any \(r= \sum r^{[1]} \otimes \cdots \otimes r^{[n]} \in V^{\otimes n} \),
\[
\sigma  \cdot r  = \sum \varepsilon(\sigma; r^{[1]}, \dots, r^{[n]}) r^{[\sigma^{-1}(1)]} \otimes \cdots \otimes r^{[\sigma^{-1}(n)]},
\]
where \( \varepsilon(\sigma; r^{[1]}, \dots, r^{[n]}) \) is the Koszul sign obtained from permuting the graded elements \( r^{[1]}, \dots, r^{[n]} \). We write $
\sigma^{-1}\cdot  r  \text{ as } r^{\sigma(1), \dots, \sigma(n)}.$

For \( 0 \leqslant i_1, \dots, i_r \leqslant n \) with \( i_1 + \cdots + i_r = n \),  let \( \Sh(i_1, i_2, \dots, i_r) \) denote the set of  \((i_1, \dots, i_r)\)-shuffles, i.e., permutations \( \sigma \in \mathfrak{S}_n \) such that:
\[
\sigma(1) <   \cdots <   \sigma(i_1),
\sigma(i_1+1) <   \cdots <   \sigma(i_1 + i_2),  \dots, 
\sigma(i_1 + \cdots + i_{r-1} + 1) <   \cdots <   \sigma(n).
\]

\medskip

\section{Rota-Baxter algebras and double Poisson algebras}\label{Section: Rota-Baxter algebras and double Poisson algebras}

\subsection{ Rota-Baxter algebras, double Lie algebras and Yang-Baxter Equations }\

In this section, we will first recall some basic notions on Rota-Baxter algebras, double Lie algebras and associative Yang-Baxter equations. Then we will recall the connections among these three objects introduced by Schedler \cite{Sched09}, Goncharov and Kolesnikov \cite{GK18}. 

\begin{defn}\label{Def: Rota-Baxter   algebra}
Let \( (A, \mu = \cdot) \) be an associative algebra over a field \( \mathbf{k} \), and let \( M \) be a bimodule over \( A \). A linear operator \( T: M \to A \) is called a \textbf{relative Rota-Baxter operator on \( M \)} if it satisfies the following relation:
\begin{equation}\label{Eq: Rota-Baxter relation}
	T(a) \cdot T(b) = T\big(a\cdot  T(b) + T(a)\cdot  b\big),
\end{equation}
for all \( a, b \in A \). In this case, the triple \( (A, M, T) \) is called a \textbf{relative Rota-Baxter algebra}.

In particular, if we take \( M = A \), then \( T \) is simply called a \textbf{Rota-Baxter operator}, and \( (A, \cdot, T) \) is called a \textbf{Rota-Baxter algebra}.
	
\end{defn}

\begin{defn}\label{Def:Double Lie alg}\cite{Sched09,VdB08}
	A {\bf double Lie algebra} is a linear space $V$ equipped with a linear map $$ \lbb-, -\rbb:V\otimes V\rightarrow V\otimes V$$ satisfying the following identities	 for all $a,b,c\in V$
	\begin{itemize}
		\item [(i)] Skew-symmetry:
		\begin{equation}\label{Eq: antisymmetric condition of double Lie system}
			\lbb a, b\rbb=-\sigma_{(12)}\lbb b, a\rbb;
		\end{equation}
		\item [(ii)]  Double Jacobi identity:
		\begin{equation}\label{Eq: cyclically condition of double Lie system}
			\lbb -,\lbb -, -\rbb\rbb_L+\sigma_{(123)}\lbb -,\lbb -, -\rbb\rbb_L\sigma_{(123)}^{-1}+\sigma_{(123)}^2\lbb -,\lbb -, -\rbb\rbb_L\sigma_{(123)}^{-2}=0.
		\end{equation}
		where $\lbb-,-\rbb_L(x_1\otimes x_2\otimes x_3):=\lbb x_1, x_2\rbb\otimes x_3.$
	\end{itemize}

\end{defn}

\begin{defn}\label{Def:Double Poisson alg}\cite{VdB08}
	A {\bf double Poisson algebra} is an associative algebra $(A,\cdot)$ equipped with a double Lie algebra structure $\lbb-,-\rbb$ satisfying the  Leibniz rule: for all $a,b,c\in A$
	\begin{eqnarray}
		\lbb a,b\cdot c\rbb =  \lbb a,b\rbb\cdot c+ b\cdot\lbb a,c\rbb,
	\end{eqnarray}
	where
	\[\lbb a,b\rbb \cdot c= \lbb a,b\rbb ^{[1]} \otimes (\lbb a,b\rbb^{[2]}\cdot c), \]
	\[b\cdot\lbb a,c\rbb =(b\cdot \lbb a,c\rbb^{[1]})\otimes   \lbb a,c\rbb^{[2]}.\]
\end{defn}

\smallskip

 Goncharov and Kolesnikov \cite{GK18} proved that double Lie algebra structures on a finite-dimensional vector space V are equivalent to cyclic Rota-Baxter operators (referred to as a skew-symmetric Rota-Baxter operators in their work) on the associative algebra $\mathrm{End}(V)$. We briefly recall this correspondence below. 
 
For a finite-dimensional vector space \( V \), there is a natural nondegenerate bilinear form $\langle  -,-\rangle $ on $\End(V)$ which is given as: \[\langle  f,g\rangle :=\mathrm{tr}(f\circ g), \forall f,g\in \End(V).\] Thus we have an isomorphism \[\End(V)\cong \End(V)^\vee,\] which induces the following isomorphisms:
\[
\End(V \otimes V) \cong \End(V) \otimes \End(V) \cong \End(V)\ot \End(V)^\vee \cong \End(\End(V)).
\]
In this way, any double bracket \[ \lbb -,-\rbb : V \otimes V \to V \otimes V \] can be uniquely determined by a linear operator \[ T : \End(V) \to \End(V).\]
Conversely, given a linear operator \( T \) on \( \End(V) \), the corresponding bracket \( \lbb -,-\rbb : V \otimes V \to V \otimes V \) can be expressed in terms of \( T \) as follows:
\begin{equation}\label{double bracket}
	\lbb a, b \rbb = \sum_{i=1}^N T^\vee(e^i)(a) \otimes e_i(b) = \sum_{i=1}^N e^i(a) \otimes T(e_i)(b), \quad a, b \in V,
\end{equation}
where $ \{e_1, \cdots, e_N\}$ is a basis of \( \End(V) \), and $\{e^1, \cdots, e^N \}$ is the corresponding dual basis with respect to the trace form, i.e., $\langle  e^i, e_j\rangle =\delta^i_j$. Here, \( T^\vee\) denotes the adjoint (or conjugate) operator of \( T \) on \( \End(V) \) with respect to the trace form. 

Goncharov and Kolesnikov proved that the bracket $\lbb-,-\rbb$ defines a double Lie algebra structure if and only if the operator $T$ is a cyclic Rota-Baxter operator  on $\End(V)$, that is, $T$ is a Rota-Baxter operator satisfying $T = -T^\vee$.

On the other hand, Schedler \cite{Sched09} established a correspondence between skew-symmetric solutions of the associative Yang-Baxter equation (AYBE) in $\End(V)$ and double Lie algebra structures on $V$.

\begin{defn}\cite{Agu00b}
	Let $A$ be a unital associative algebra. An element $r = \sum_i a_i \otimes b_i \in A \otimes A$ is called a solution to the \textbf{associative Yang-Baxter equation (AYBE)} in $A$ if
	\[
	\operatorname{AYBE}(r) := r_{12} \cdot r_{13} - r_{23} \cdot r_{12} + r_{13} \cdot r_{23} = 0 
	\]
	in $A \otimes A \otimes A$, where the tensors $r_{12}$, $r_{13}$, and $r_{23}$ are given by
	\[
	r_{12} = \sum_i a_i \otimes b_i \otimes 1,\quad 
	r_{13} = \sum_i a_i \otimes 1 \otimes b_i,\quad 
	r_{23} = \sum_i 1 \otimes a_i \otimes b_i.
	\]
	A solution $r$ is said to be \emph{skew-symmetric} if $r = -r^{21}$, where $r^{21} = \sum_i b_i \otimes a_i$.
\end{defn}

Now let $A = \End(V)$ for a vector space $V$. Then there is a canonical isomorphism
\[
\End(V) \otimes \End(V) \cong \End(V \otimes V),
\]
under which each element $r = \sum_i a_i \otimes b_i$ corresponds to a unique bilinear operation
\[
\lbb - , - \rbb : \End(V) \otimes \End(V) \rightarrow \End(V) \otimes \End(V).
\]
Schedler proved that an element $r \in \End(V) \otimes \End(V)$ is a skew-symmetric solution of the associative Yang-Baxter equation in $\End(V)$ if and only if the associated double bracket $\lbb - , - \rbb$ defines a double Lie algebra structure on $V$.

In summary, we have the following equivalence:

\begin{thm}\cite{GK18,Sched09}
	Let $V$ be a finite-dimensional vector space. The following data are equivalent:
	\begin{itemize}\label{RBS=Double Lie system}
		\item[\rm (i)] A double Lie algebra structure $\lbb - , - \rbb$ on $V$.
		
		\item[\rm (ii)] A linear operator $T : \End(V) \rightarrow \End(V)$ that is a Rota-Baxter operator with respect to composition (i.e., on $(\End(V), \circ)$), and is cyclic, meaning $T^\vee = -T$.
		
		\item[\rm (iii)] A skew-symmetric solution $r \in \End(V \otimes V)$ of the associative Yang-Baxter equation, i.e., $r = -r^{21}$ and $\operatorname{AYBE}(r) = 0$.
	\end{itemize}
\end{thm}

\bigskip

\section{ Pre-Calabi-Yau algebras and homotopy Rota-Baxter algebras}\label{ Cyclic homotopy relative Rota-Baxter algebras}

In this section, we begin by recalling key concepts related to pre-Calabi-Yau algebras, following the work of Fern\'andez and Herscovich \cite{FH21}, including cyclic $A_\infty$-algebras, good, manageable, and special pre-Calabi-Yau algebras. We then review the notions of homotopy Rota-Baxter algebras and homotopy relative Rota-Baxter algebras. Building on these, we introduce the concepts of absolute and relative cyclic and ultracyclic homotopy relative Rota-Baxter algebras—homotopy Rota-Baxter structures that satisfy certain cyclic invariance conditions. These structures will play a central role in the remainder of the paper. Finally, we present a construction method for cyclic homotopy Rota-Baxter algebras, referred to as cyclic completion.

	
	\subsection{Cyclic $A_\infty$-algebras and Pre-Calabi-Yau algebras}\

	We first recall some basics on $A_\infty$-algebras and  $A_\infty$-bimodules. 
	
	\begin{defn}
		Let  $A=\oplus_{n \in \mathbb{Z}} A_n$ be  a graded vector space. If $A$ is  equipped with a family of homogeneous
		linear maps $\{m_n:A^{\otimes n}\to A\}_{n\geqslant1}$, with $|m_n|=n-2$ satisfying the Stasheff identity: for all $n\geqslant1$,
		\begin{eqnarray}\label{Eq: stasheff-id}
			\sum\limits_{    i+j+k= n,\atop
				i, k\geqslant 0, j\geqslant 1 } (-1)^{i+jk}m_{i+1+k}\circ\Big(\id^{\ot i}\ot m_j\ot \id^{\ot k}\Big)=0,
		\end{eqnarray}
		then $(A,\{m_n\}_{n\geqslant1})$ is called an {\bf $A_\infty$-algebra}.
	\end{defn}

	
	\begin{defn}
		Let $(A,\{m_n\}_{n\geqslant1})$ be an $A_\infty$-algebra. An {\bf $A_\infty$-bimodule} over $A$ is a graded space $M=\bigoplus\limits_{n \in \mathbb{Z}} M_n$ equipped with a family of homogeneous maps $\{m_{p,q}:A^{\ot p}\ot M\ot A^{\ot q}\to M\}_{p,q\geq0}$ with $|m_{p,q}|=p+q-1$ satisfying: for all $p,q\geqslant 0 $,	
		\begin{align}\label{Eq: A-infty-bimodule}
			 &\sum\limits\limits_{1\leqslant j\leqslant p\atop 0\leqslant i\leqslant p-j} (-1)^{i+j(p-i-j+1+q)}m_{p-j+1,q}\circ\Big(\id^{\ot i}\ot m_j\ot \id^{\ot p-i-j}\ot \id_M\ot  \id^{\ot q}\Big)\\
			\nonumber &+\sum\limits\limits_{i+r=p,s+k=q\atop i,r,s,k\geqslant 0} (-1)^{i+(r+s-1)k+1}m_{i+1,k+1}\circ\Big(\id^{\ot i}\ot m_{r,s}\ot \id^{\ot k}\Big)\\
			\nonumber &+\sum\limits\limits_{ 1\leqslant j\leqslant q\atop 0\leqslant i\leqslant q-j } (-1)^{p+i+1+j(q-i-j)}m_{p,q-j+1}\circ\Big(\id^{\ot p}\ot \id_M\ot \id^{\ot i}\ot m_j\ot \id^{\ot q-i-j}\Big)\nonumber\\
		\nonumber	&=0.
		\end{align}
		
	\end{defn}

	\begin{defn}Let $d$ be an integer. Let $A$ be a graded space and $\gamma: A\times A \rightarrow \bf{k}$ be a graded symmetric bilinear form of degree $-d$ .
		An operation $m_n: A^{\ot n}\rightarrow A$ is called {\bf$d$-cyclic} with respect to $\gamma$ if it satisfies 
		$$
		\gamma\left(m_n\left(a_1\otimes \cdots\otimes a_n\right), a_0\right)=(-1)^{n+\left|a_0\right|(\sum\limits_{i=1}^n\left|a_i\right|)} \gamma\left(m_n\left(a_0\otimes\cdots\otimes  a_{n-1}\right), a_n\right),
		$$
		for all homogeneous elements $a_0, \cdots, a_n \in A$. 
	\end{defn}
	
	\begin{defn}
		Let $d \in \mathbb{Z}$. A \textbf{$d$-cyclic $A_\infty$-algebra} is an $A_\infty$-algebra $(A, \{m_n\}_{n \geqslant 1})$ equipped with a graded symmetric, nondegenerate bilinear form $\gamma: A \times A \to \bfk$, such that each $m_n$ is $d$-cyclic with respect to $\gamma$.
	\end{defn}
	
%
%
%
	
	\begin{remark}Actually, a $d$-cyclic $A_\infty$-structure on $A_\infty$-algebra $A$ is equivalent to a strict $A_\infty$-bimodule isomorphism from $A$ to $s^{d}A^\vee$.	\end{remark}

	Let $d\in \mathbb{Z}$.  Set \[\partial_d A=A \oplus  s^{d}A^{\vee} .\]  There is a natural bilinear form \[\zeta_A: \partial_d A\times \partial_d A\rightarrow \bf{k}\] of degree $-d$ on $\partial_d A$ defined as follows:
	$$
	\zeta_A(s^{d} f, a)=(-1)^{|a|(|f|+d)} \zeta_A(a, s^{d}f)=f(a), \ \  \zeta_A(a, b)=\zeta_A(s^{d} f, s^{d} g)=0,
	$$
	for all homogeneous $a, b \in A$ and $f, g \in A^{\vee}$. Note that $\zeta_A$ has degree $-d$. Moreover, if $A$ is an $A_\infty$-algebra, then $\partial_dA$ has a natural $A_\infty$-algebra structure, i.e., the trivial extension $A_\infty$-algebra.
	
	\begin{prop}Let $A$ be an $A_\infty$-algebra. Then $\partial_d A$ is a cyclic $A_\infty$-algebra of degree $d$ with respect to the bilinear form $\zeta_A$.
	\end{prop}

	\begin{defn}\cite{KTV25}
		Let  $d\in \mathbb{Z}$. A \textbf{$d$-pre-Calabi-Yau} structure on a graded space $A=\oplus_{n \in \mathbb{Z}} A_n$ consists of  a  $(d-1)$-cyclic   $A_{\infty}$-algebra structure $\{m_n\}_{n\geqslant1}$ on   $ \partial_{d-1} A=A \oplus s^{d-1}A^{\vee}  $  with respect to the natural bilinear form $\zeta_A: \partial_{d-1} A \otimes \partial_{d-1} A \rightarrow \bfk$ such that $m_n(A^{\otimes n})\subset A$ for all $n\geqslant 1$; that is, $\partial_{d-1}A$ contains  $A$ as an $A_\infty$-subalgebra.

		A 0-pre-Calabi-Yau algebra will be simply called a   {\bf pre-Calabi-Yau algebra} .
	\end{defn}
	
	\begin{remark}
		In the original definition of pre-Calabi-Yau algebras by Kontsevich, Takeda, and Vlassopoulos in \cite{KTV25}, a pre-Calabi-Yau algebra is defined as a space endowed with a complicated family of operations involving multiple inputs and outputs, subject to certain compatibility conditions. They proved that, on a finite-dimensional space, a pre-Calabi-Yau structure in this sense is equivalent to the one described above.
	\end{remark}

We now introduce certain pre-Calabi–Yau algebras satisfying specific desirable properties, following primarily \cite{FH21}.

\begin{defn} Let $A$ be a pre-Calabi-Yau algebra. We say that $A$ is
	\begin{itemize}
		\item[(i)] {\bf good} if the $A_{\infty}$-algebra structure $\{m_n\}_{n\geqslant1}$ on   $ \partial_{-1} A=A \oplus s^{-1}A^{\vee}  $ satisfies: 
		
		\begin{itemize}
			\item [(a)]for all $i> 1$, $m_{2i}=0$;
			\item  [(b)] for   all $i\geqslant 1$, 
			\begin{align*}
				& m_{2i-1}(A\otimes s^{-1}A^{\vee}\otimes\cdots\otimes s^{-1}A^{\vee}\otimes A)\subseteq A, \\
			&m_{2i-1}( s^{-1}A^{\vee}\otimes A\otimes\cdots\otimes A\otimes s^{-1}A^{\vee})\subseteq s^{-1}A^{\vee},
			\end{align*}
			and $m_{2i-1}$  vanishes in all other cases;
		\end{itemize}   
		
		\item[(ii)] \label{Def: fine} {\bf fine} if it is good and $m_2$ also vanishes.
		\item [(iii)]\label{Def: manageable} {\bf manageable} if $m_2$ restricted to $A$ is an associative multiplication, denoted by $``\cdot"$, and for $(a,s^{-1}f),(b,s^{-1}g)\in \partial_{-1} A$:
		\[m_2((a,s^{-1}f)\otimes (b,s^{-1}g)=(a\cdot b,(-1)^{|a|}s^{-1}a{\rhd}g+s^{-1}f{\lhd}b);\]
		where the symbols ``$\rhd$" and ``$\lhd$" denote the natural left and right actions of $A$ on $A^\vee$ respectively, induce by the multiplication on $A$.
		\item[(iv)] {\bf special}  if the $A_{\infty}$-algebra structure $\{m_n\}_{n\geqslant1}$ on   $ \partial_{-1} A=A \oplus s^{-1}A^{\vee}  $ satisfies:  for all $n>1$ $m_{2n-1}$ is ultracyclic, that is, $v_1,\cdots v_{2n}\in \partial_{-1} A$, and $\sigma\in \mathfrak{S}_{n}$,
		\[\zeta_A(m_{2n-1}(v_1\otimes\cdots\otimes v_{2n-1}), v_{2n})=\varepsilon(\widetilde{\sigma}; v_1,v_2, \dots, v_{2n-1}, v_{2n})\zeta_A(m_{2n-1}(v_{\widetilde{\sigma}(1)}\otimes \cdots\otimes v_{\widetilde{\sigma}(2n-1)}), v_{\widetilde{\sigma}(2n)}),\]
		where $\widetilde{\sigma}\in \mathfrak{S}_{2n}$ is defined as $\widetilde{\sigma}(2i)=2\sigma(i)$, $ \widetilde{\sigma}(2i-1)=2\sigma(i)-1$ for all $1\leqslant i\leqslant n$.
	\end{itemize}
	\end{defn}

	\medskip

%
%
	
	\subsection{Homotopy Rota-Baxter algebras}\ 
	
	In this subsection, we review the notions of homotopy Rota-Baxter algebras and relative Rota-Baxter algebras. We then introduce the concepts of homotopy Rota-Baxter modules, along with a trivial extension construction that produces homotopy Rota-Baxter algebras from such modules. Finally, we present the concept of cyclic homotopy Rota-Baxter algebras—a distinguished class of homotopy Rota-Baxter algebras endowed with a desirable cyclic invariance property—and provide a canonical construction of these structures.
	
\begin{defn} \cite{WZ24}\label{Def: absolute homotopy Rota-Baxter algebras}
	Let $(A,\{m_n\}_{n\geqslant 1})$ be an $A_\infty$-algebra. A homotopy Rota-Baxter operator consists of a family of operators $\{T_n: A^{\ot n}\rightarrow A\}_{n\geqslant 1}$ with $|T_n|=n-1$ subjecting to the following identities for all $n\geqslant 0$:
	\begin{eqnarray}\label{Equation: absolute homotopy Rota-Baxter algebras}
			&&\sum\limits\limits_{ l_1+\dots+l_k=n,\atop
			l_1, \dots, l_k\geqslant 1 } (-1)^{\delta}m_k\circ\Big(T_{l_1}\ot \cdots \ot T_{l_k}\Big)\\
		\nonumber &=&\sum\limits\limits_{1\leqslant j\leqslant p}\sum\limits\limits_{  r_1+\dots+r_p=n,\atop
			r_1, \dots, r_p\geqslant 1 }  (-1)^\eta
		T_{r_1}\circ\Big(\id^{\ot i}\ot m_{p}\circ( T_{r_2}\ot   \cdots\ot T_{r_j}\ot\id \ot T_{r_{j+1}}\ot \cdots\ot T_{r_p})\ot \id^{\ot k}\Big)
		\end{eqnarray}
		where
		\begin{small}
		\begin{align*}\delta=\frac{k(k-1)}{2}+\sum\limits_{j=1}^k(k-j)l_j,\ \  \eta&=i+(p+\sum\limits\limits_{j=2}^p(r_j-1))k+\sum\limits_{t=2}^j(r_t-1)+\sum\limits_{t=2}^p(r_t-1)(p-t).
		\end{align*}
		\end{small}
The triple $(A,\{m_n\}_{n\geqslant 1},\{T_n\}_{\geqslant 1})$ is called a \textbf{homotopy Rota-Baxter algebra}.
	\end{defn}

	We also need the concepts of modules over homotopy Rota-Baxter algebras.
	\begin{defn}\label{Definition: Rota-Baxter module}
		Let $(A,\{m_n\}_{n\geqslant 1},\{T_n\}_{n\geqslant 1})$ be a homotopy Rota-Baxter algebras. A Rota-Baxter module over $A$ is an $A_\infty$-bimodule $(M,\{m_{i,j}\}_{i,j\geqslant 0})$ over $A$ which is endowed with a family of graded maps $\{T^M_{i,j}:A^{\otimes i}\otimes M\otimes A^{\otimes j}\to M\}_{i,j\geqslant 0}$, with $|T^M_{i,j}|=i+j$, such that the following identities hold for any $m,n\geqslant 0$:
		\begin{small}
\begin{eqnarray}\label{Defn: homotopy Rota-Baxter module} 
			&&\sum\limits\limits_{{ i_1+\dots+i_p+l=m,\atop j_1+\cdots+j_q+k=n}\atop
			 p,q,l,k\geqslant 0 } (-1)^{\alpha}m_{p,q}\circ\Big(T_{i_1}\ot \cdots \ot T_{i_p}\ot T^M_{l,k} \ot T_{j_1}\otimes \cdots \ot T_{j_q}\Big)\\
			 &=&\nonumber\sum\limits\limits_{{{{i_1+\cdots+i_p+l=m,\atop j_1+\cdots+j_q+k=n,}}\atop i_1,\cdots,i_p,j_1,\cdots,j_q\geqslant 1}\atop p,\ q,\ l,\ k\geqslant0;} (-1)^{\beta_1}
			T^M_{l,k}\circ\Big(\id_A^{\ot l}\ot m_{p, q}\circ( T_{i_1}\ot   \cdots\ot T_{i_p}\ot\id_M \ot T_{j_{1}}\ot \cdots\ot T_{j_q})\ot \id_A^{\ot k}\Big)\\
			&&\nonumber+\sum\limits_{{{i_1+\cdots+i_p+l+r+1=m\atop j_1+\cdots+j_q+k+t=n} \atop i_1,\cdots,i_p,j_1,\cdots,j_q\geqslant0 }  \atop v,l,k,p,q\geqslant 0}  (-1)^{\beta_2}
			T_{l,k}^M\circ\Big(\id_A^{\ot l}\ot m_{p+1,q}\circ( T_{i_1}\ot \cdots \ot T_{i_v}\ot \id_A\ot T_{i_{v+1}}  \cdots\\
			&&\nonumber \quad\quad\quad\quad\quad\quad\quad   \cdots\ot T_{i_p}\ot T^M_{r,t} \ot T_{j_{1}}\ot \cdots\ot T_{j_q})\ot \id_A^{\ot k}\Big)\\
			&&\nonumber+\sum\limits_{{{i_1+\cdots+i_p+l+r=m\atop j_1+\cdots+j_q+k+t+1=n} \atop i_1,\cdots,i_p,j_1,\cdots,j_q\geqslant0 }  \atop v,l,k,p,q\geqslant 0}  (-1)^{\beta_3}
			T_{l,k}^M\circ\Big(\id_A^{\ot l}\ot m_{p, q+1}\circ( T_{i_1}\ot   \cdots\ot T_{i_p}\ot T^M_{r,t} \ot T_{j_{1}}\ot \cdots \\
			&&\nonumber \quad\quad\quad\quad\quad\quad\quad\cdots \ot T_{j_v}\ot \id_A\ot T_{j_{v+1}}\ot\cdots  T_{j_q})\ot \id_A^{\ot k}\Big),
		\end{eqnarray}
		\end{small}
		where 	
	\begin{small}	\begin{align*}
			\alpha =&\frac{(p+q)(p+q+1)}{2}+q(l+k)+\sum\limits_{t=1}^q(q-t)j_t+\sum\limits_{t=1}^p(p+q+1-t)i_t,\\
			 \beta_1=&  l+k(m+n-l)+\sum\limits_{t=1}^p(i_t-1)+\sum\limits_{t=1}^p(i_t-1)(p+q-t)+\sum\limits_{t=1}^q(j_t-1)(q-t),\\
			 \beta_2=&  l+k(m+n-l)+\sum\limits_{s=1}^v(i_s-1)+(r+t)q+\sum\limits_{s=1}^p(i_s-1)(p+q+1-t)+\sum\limits_{s=1}^q(j_s-1)(q-t),\\
			 \beta_3=&  l+k(m+n-l)+\sum\limits_{s=1}^p(i_s-1)+\sum\limits_{s=1}^v(j_s-1)+(r+t)(q-1)+\sum\limits_{s=1}^p(i_s-1)(p+q+1-t)+\sum\limits_{s=1}^q(j_s-1)(q-t).
		\end{align*}
		\end{small}
		\end{defn}	
		
\begin{defn}\label{Def:HRBA} \cite{DM20}
	Let $(A, \{m_i\}_{i \geqslant1})$ be an $A_\infty$-algebra and $(M, \{m_{p,q}\}_{p,q \geq 0})$ an $A_\infty$-bimodule over $A$. A \textbf{homotopy relative Rota-Baxter operator} on $M$ is a family of operators $\{T_n: M^{\otimes n} \to A\}_{n \geqslant1}$ of degree $|T_n| = n - 1$ satisfying the following identity for all $n \geqslant1$:
	\begin{eqnarray} \label{Eq: homotopy RB-operator-version-2}
		&&\sum\limits\limits_{ l_1+\dots+l_k=n,\atop
			l_1, \dots, l_k\geqslant 1 } (-1)^{\delta}m_k\circ\Big(T_{l_1}\ot \cdots \ot T_{l_k}\Big)\\
		\nonumber &=&\sum\limits\limits_{1\leqslant j\leqslant p}\sum\limits\limits_{  r_1+\dots+r_p=n,\atop
			r_1, \dots, r_p\geqslant 1 }  (-1)^\eta
		T_{r_1}\circ\Big(\id^{\ot i}\ot m_{j-1, p-j}\circ( T_{r_2}\ot   \cdots\ot T_{r_j}\ot\id \ot T_{r_{j+1}}\ot \cdots\ot T_{r_p})\ot \id^{\ot k}\Big),
	\end{eqnarray}
	where the signs $\delta$ and $\eta$ are as defined in Definition~\ref{Def: absolute homotopy Rota-Baxter algebras}. The triple $(A,M,\{T_i\}_{i\geqslant1})$ is called a {\bf homotopy relative Rota-Baxter algebra}.
	 
\end{defn}

In particular, when the underlying $A_\infty$-algebra and $A_\infty$-bimodule of a homotopy relative Rota-Baxter algebra are simply a differential graded (dg) algebra and a dg bimodule over the dg algebra, respectively, the notion simplifies as follows. This special case will be used in later sections.

\begin{defn}
	Let $(A, d, m)$ be a dg algebra, and let $(M, d_M, m^l, m^r)$ be a dg $A$-bimodule, where $m^l$ and $m^r$ denote the left and right actions of $A$, respectively. A \textbf{homotopy relative Rota-Baxter operator} on $M$ is a family of operations $\{T_n: M^{\otimes n} \to A\}_{n \geqslant 1}$ with $|T_n|=n-1$ satisfying the identity
	\begin{align} \label{Eq: homotopy dg RB-operator}
		d \circ T_n 
		&- \sum_{s + k + 1 = n} (-1)^{n-1} T_n \circ (\id^{\otimes s} \otimes d_M \otimes \id^{\otimes k}) \\
		\nonumber
		&= - \sum_{i + j = n} (-1)^{1+i} \, m \circ (T_i \otimes T_j) \\
		\nonumber
		&\quad + \sum_{i + j + k + 1 = n} (-1)^{i + (j-1)(k+1)} T_{i+k+1} \circ (\id^{\otimes i} \otimes m^l \circ (T_j \otimes \id) \otimes \id^{\otimes k}) \\
		\nonumber
		&\quad + \sum_{i + j + k + 1 = n} (-1)^{i + (j-1)k} T_{i+k+1} \circ (\id^{\otimes i} \otimes m^r \circ (\id \otimes T_j) \otimes \id^{\otimes k})
	\end{align}
	for all $n \geqslant 1$. In this case, the triple $(A, M, \{T_n\}_{n \geq 1})$ is called a \textbf{dg homotopy relative Rota-Baxter algebra}.
\end{defn}

%

 \begin{remark}
 	\begin{itemize}
 		\item[(i)] Given a homotopy Rota-Baxter algebra $(A, \{m_n\}_{n \geq 1}, \{T_n\}_{n \geq 1})$, the space $A$ itself naturally forms a homotopy Rota-Baxter module over $A$. Explicitly, the structure maps are given by $m^A_{p,q} = m_{p+q+1}$ and $T^A_{p,q} = T_{p+q+1}$.
 		
 		\item[(ii)] In Equation~\eqref{Equation: absolute homotopy Rota-Baxter algebras} of Definition~\ref{Def: absolute homotopy Rota-Baxter algebras}, if we replace one instance of $A$ in the inputs with a module $M$, and correspondingly replace the operations $m_n$ and $T_n$ with $m_{p,q}$ and $T_{p,q}$ to reflect the presence of $M$, we recover Equation~\eqref{Defn: homotopy Rota-Baxter module}. Furthermore, if all instances of $A$ in the inputs are replaced by $M$, we obtain Equation~\eqref{Eq: homotopy RB-operator-version-2} from the definition of homotopy relative Rota-Baxter algebras.
 	\end{itemize}
 \end{remark}

\begin{prop}
	\label{Proposition:Trivial Extension of homotopy Rota-Baxter algebras}
	Let $(A,\{m_n\}_{n\geqslant 1},\{T_n\}_{n\geqslant 1})$ be a homotopy Rota-Baxter algebra, and let $(M,\{m_{i,j}\}_{i,j\geqslant 0},\{T_{i,j}\}_{i,j\geqslant 0})$ be a homotopy Rota-Baxter module over $A$. Then there exists a canonical homotopy Rota-Baxter algebra structure 
	\[
	\left(\{\widetilde{m}_n\}_{n\geqslant 1},\{\widetilde{T}_n\}_{n\geqslant 1}\right)
	\]
	on the graded space $A \oplus M$, where the structure maps
	\[
	\widetilde{m}_n : (A \oplus M)^{\otimes n} \to A \oplus M \quad \text{and} \quad \widetilde{T}_n : (A \oplus M)^{\otimes n} \to A \oplus M
	\]
	are defined as follows:
	\begin{align*}
		\widetilde{m}_n|_{A^{\otimes n}} &= m_n, &
		\widetilde{m}_n|_{A^{\otimes i} \otimes M \otimes A^{\otimes j}} &= m_{i,j}, \\
		\widetilde{T}_n|_{A^{\otimes n}} &= T_n, &
		\widetilde{T}_n|_{A^{\otimes i} \otimes M \otimes A^{\otimes j}} &= T_{i,j},
	\end{align*}
	where $i + j + 1 = n$. The maps $\widetilde{m}_n$ and $\widetilde{T}_n$ vanish on all other components of $(A \oplus M)^{\otimes n}$. This homotopy Rota-Baxter algebra, denoted by $A \ltimes M$, is called the trivial extension of $A$ by $M$.
\end{prop}
	\begin{proof}
		This is just the analog of classical trivial extension of $A_\infty$-algebras by $A_\infty$-bimodules. It can be checked by direct computations, so we omit the details here.
		\end{proof}
		
		\begin{prop}\label{Proposition: dual homotopy Rota-Baxter module}Let $(M,\{m_{i,j}\}_{i,j\geqslant 0},\{T_{i,j}\}_{i,j\geqslant 0})$ be a homotopy Rota-Baxter module over homotopy Rota-Baxter algebras $(A,\{m_n\}_{n\geqslant 1},\{T_n\}_{n\geqslant 1})$. Then $M^\vee$ has a canonical homotopy Rota-Baxter module structure, in which $T_{i,j}^{M^\vee},m_{i,j}^{M^\vee}: A^{\ot i}\otimes M^\vee\otimes A^{\ot j}\rightarrow M^\vee $ are defined as follows: 
			\begin{align*}&m_{i,j}^{M^\vee}(a_1\otimes \cdots  a_i\otimes f\otimes b_{1}\otimes \cdots \otimes b_{j})(x)\\
				=&(-1)^{(j+1)(i+j+1)+ (\sum\limits_{k=1}^i|a_k|)(|f|+|x|+\sum\limits_{k=1}^j|b_k|)+|f|(i+j-1)} f\big(m_{j,i}^M(b_{1}\otimes \cdots \otimes b_j\otimes x\otimes a_1\otimes \cdots \otimes a_i)\big)\\
				&\\
				&T_{i,j}^{M^\vee}(a_1\otimes \cdots  a_i\otimes f\otimes b_{1}\otimes \cdots \otimes b_{j})(x)\\
				=&(-1)^{(j+1)(i+j+1)+ (\sum\limits_{k=1}^i|a_k|)(|f|+|x|+\sum\limits_{k=1}^j|b_k|)+|f|(i+j)} f\big(T_{j,i}^M(b_{1}\otimes \cdots \otimes b_j\otimes x\otimes a_1\otimes \cdots \otimes a_i)\big).
				\end{align*}
In particular, $A^\vee$ is a homotopy Rota-Baxter module over $A$.
			\end{prop}
			\begin{proof}
				The proof of this proposition involves extensive computations. For the sake of readability, we have placed the proof in the Appendix~\ref{Appendix:  Proof of dual homotopy Rota-Baxter module}.
				\end{proof}

	{ \begin{defn}Let $A$ be a cyclic $A_\infty$-algebra with respect to a nondegenerate bilinear form $\gamma: A\ot A \rightarrow \bfk$. A homotopy Rota-Baxter operator $\{T_n\}_{n\geqslant 1}$ on $A$ is said to be  cyclic if each operator $T_n:A^{\ot n}\rightarrow A$ is cyclic with respect to the bilinear form $\gamma$. Then $(A,\{T_n\}_{n\geqslant 1})$ is called a {\bf cyclic  homotopy (absolute) Rota-Baxter algebras}.
	
	 Moreover, we call $\{T_n\}_{n \geqslant 1}$ an \textbf{ultracyclic homotopy Rota–Baxter operator} if each operator $T_n$ is both cyclic and skew-symmetric, i.e., for all $\sigma \in \mathfrak{S}_n$, the identity
	 \[T_n \circ \sigma = \operatorname{sgn}(\sigma) T_n \]
	 holds. In this case, the pair $(A, \{T_n\}_{n \geqslant 1})$ is called an \textbf{ ultracyclic homotopy Rota–Baxter algebra.}

	\end{defn}
}

	We give a method to construct  the cyclic   homotopy Rota-Baxter algebras from homotopy Rota-Baxter algebras,  called the {\bf cyclic completion for homotopy Rota-Baxter algebras}.
	\begin{prop}\label{Prop:From homotopy Rota-Baxter algebras to cyclic homotopy Rota-Baxter algerbas}
		Let $(A,\{m_n\}_{n\geqslant 1},\{T_n\}_{n\geqslant1})$ be a locally finite-dimensional homotopy Rota-Baxter algebra. Then $\partial_0A:=A\ltimes A^\vee$ is a cyclic homotopy Rota-Baxter algebra.  Precisely, the homotopy Rota-Baxter operator $\{\widetilde{T}_n\}_{n\geqslant 1}$ is given by the following formulas:
 for homogeneous elements $(a_1,f_1),\cdots,(a_n,f_n)\in \partial_0A=A\oplus A^\vee$, $$\widetilde{T}_n:(\partial_0A)^{\otimes n}\longrightarrow \partial_0A$$
	
	\[	\widetilde{T}_n\left((a_1,f_1)\otimes \cdots \otimes (a_n,f_n)\right)=\left(T_n(a_1\otimes \cdots\otimes a_n),\sum\limits_{j=1}^n(-1)^{\xi}  f_j\circ T_n(a_{j+1}\otimes \cdots\otimes a_n\otimes -\otimes a_1\otimes \cdots\otimes a_{j-1})\right),\]
	where
	\begin{small} \[\xi=jn+(n-1)|f_j|+(\sum\limits_{k=1}^{j-1}(|a_k|))(|f_j|+\sum\limits_{k=j+1}^n(|a_k|)).\]
		\end{small}
 Moreover, if $\{T_n\}_{n\geqslant1}$  is skew-symmetric, then   $\partial_0A$ is an ultracyclic homotopy Rota-Baxter algebra.
	\end{prop}
	\begin{proof}
		According to  Proposition~\ref{Proposition:Trivial Extension of homotopy Rota-Baxter algebras} and Proposition~\ref{Proposition: dual homotopy Rota-Baxter module}, we have that $\partial_0 A$ is a homotopy Rota-Baxter algebra, and it can be seen that this homotopy Rota-Baxter structure on $\partial_0A$ is cyclic with respect to the natural bilinear form on $\partial_0A$. According to the formulas of $\widetilde{T}_n$ presented above, one can see that $\widetilde{T}_n$ is skew-symmetric if each $T_n$ is skew-symmetric.
	\end{proof}


	We also have the notion of relative cyclic homotopy Rota-Baxter operators.

	\begin{defn}
		Let $A$ be an $A_\infty$-algebra, and let $\{T_n: (A^\vee)^{\otimes n} \rightarrow A\}_{n \geqslant 1}$ be a homotopy relative Rota-Baxter operator on the dual bimodule $A^\vee$. We say that $\{T_n\}_{n \geqslant 1}$ is a \textbf{cyclic homotopy relative Rota-Baxter operator} if, for all $n \geqslant 1$ and homogeneous elements $f_0, \dots, f_n \in A^\vee$, the following identity holds:
		\[
		\langle T_n(f_0 \otimes \cdots \otimes f_{n-1}), f_n \rangle 
		= (-1)^{n + |f_n|(\sum_{j=0}^{n-1} |f_j|)} 
		\langle T_n(f_n \otimes f_0 \otimes \cdots \otimes f_{n-2}), f_{n-1} \rangle,
		\]
		where $\langle -,- \rangle: A \times A^\vee \rightarrow \mathbf{k}$ denotes the natural pairing. Then $(A,A^\vee, \{T_n\}_{n\geqslant 1})$ is called a \textbf{cyclic homotopy relative Rota-Baxter algebra}.
		
		Moreover,  we call $\{T_n\}_{n\geqslant 1}$  an {\bf ultracyclic homotopy relative Rota-Baxter operator} if  each operator $T_n$ is cyclic and skew-symmetric, that is , each $T_n$ satisfies \[T_n\circ \sigma=\sgn(\sigma)T_n,\]
		for all  $ \sigma \in \mathfrak{S}_{n}$. In this case,  $(A, A^\vee,\{T_n\}_{n\geqslant 1})$  is called an {\bf ultracyclic homotopy relative Rota-Baxter algebra}.
	\end{defn}
	
	The above two notions, cyclic   absolute homotopy Rota-Baxter algebras and cyclic homotopy  relative Rota-Baxter algebras are related by the following construction.
	
	\begin{prop}\label{Prop: From relative cyclic RB to absolute cyclic RB}Let $A$ be a locally finite-dimensional  $A_\infty$-algebra and $\{T_n : (A^\vee)^{\ot n}\rightarrow A\}_{n\geqslant 1}$   a cyclic homotopy relative Rota-Baxter operator. Define 
		\[\overline{T}_n: (\partial_0A)^{\ot n}\twoheadrightarrow (A^\vee)^{\ot n}\xrightarrow{T_n} A\hookrightarrow \partial_0A.\]
		Then $(\partial_0A, \{\overline{T}_n\}_{n\geqslant 1})$ is a cyclic absolute homotopy Rota-Baxter algebra. 
	\end{prop}
	\begin{proof}This can be proved by direct computations, so we omit the details. 
		\end{proof}
		
\begin{remark}
	Every cyclic homotopy absolute Rota-Baxter algebra $(A, \{T_n\}_{n \geqslant 1})$ can naturally be regarded as a cyclic homotopy relative Rota-Baxter algebra $(A, A^\vee, \{T_n'\}_{n \geqslant 1})$, where each $T_n'$ is defined as the composition:
	\[T_n' : (A^\vee)^{\otimes n} \xrightarrow{\varphi^{\otimes n}} A^{\otimes n} \xrightarrow{T_n} A,\]
	where $\varphi : A^\vee \to A$ is the $A_\infty$-bimodule isomorphism induced by the non-degenerate bilinear form $\gamma$ that defines the cyclic $A_\infty$-structure on $A$.
\end{remark}

	\bigskip

	\section{Pre-Calabi-Yau structures arising from cyclic homotopy Rota-Baxter algebras }\label{Section: Pre-Calabi-Yau structures arising from cyclic homotopy Rota-Baxter algebras}

	In this section, we construct pre-Calabi-Yau algebras  from  homotopy Rota-Baxter algebras. We begin by introducing the notion of interactive pairs, consisting of two dg algebras--referred to as the acting algebra and base algebra--equipped with mutually interacting module structures that satisfy a key compatibility condition. We then demonstrate that if the acting algebra of an interactive pair is endowed with a cyclic   homotopy relative Rota-Baxter algebra satisfying certain additional conditions, then the base algebra naturally inherits a pre-Calabi-Yau  algebra structure. In particular, a dg module over a dg algebra which is endowed with a  homotopy relative Rota-Baxter algebra structure naturally inherits a pre-Calabi-Yau algebra structure.

 	\subsection{ Interactive pairs and relative derivatives.}
 \begin{defn}\label{Definition: Interactive Pair}
 	An {\bf interactive pair} $(A,B)$ consists of the following data: 
 	\begin{itemize}
 		\item[(i)]A pair of dg algebras $(A,d_A,\cdot)$ and  $(B,d_B,*)$ .
 		\item[(ii)] A left dg $B$-module structure on the complex $(A,d_A)$ and a left dg $A$-module structure on the complex $(B,d_B)$. To distinguish between them, the left action of $A$ on $B$ is  denoted by $\rhd$, while the left action of $B$  on $A$ is denoted by $\btr$. 
 		\item[(iii)] A compatibility condition ensuring that for all $a \in A, b_1,b_2\in B$, the following identity holds:
 		$$ (b_1 {\btr} a){\rhd} b_2=b_1 * (a{\rhd} b_2 ).$$
 	\end{itemize}
 We call $A$ the acting algebra of the interactive pair and $B$ the base algebra of the interactive pair.
 \end{defn}


 \begin{exam}\label{Exm: interactive pair}\
 	\begin{itemize}
 		\item[(1)] Let $A$ be a dg algebra. Then $(A,A)$ is a  interactive pair.
 		
 		\item[(2)] Let $A$ be a dg algebra and $B$ a dg $A$-module. By viewing $B$ as a dg algebra with trivial multiplication and $A$ as a $B$-module with trivial action, the pair $(A, B)$ forms an interactive pair.
 		
 		\item[(3)]Let \((B, \cdot)\) be a dg algebra. The graded vector space \(\End(B)\) carries a natural dg algebra structure, with multiplication given by composition. The algebra \(B\) becomes a left dg \(\End(B)\)-module in the canonical way. For each element \(b \in B\), define \(l_b \in \End(B)\) by \(l_b(x) := b \cdot x\) for all \(x \in B\). This gives rise to a left action of \(B\) on \(\End(B)\) defined by
 		\[
 		b \btr f := l_b \circ f,
 		\]
 		which equips \(\End(B)\) with the structure of a left dg \(B\)-module. Moreover, for all \(b_1, b_2 \in B\) and \(f \in \End(B)\), we have
 		\[
 		(l_{b_1} \circ f)(b_2) = b_1 \cdot f(b_2).
 		\]
 		Hence, \((\End(B), B)\) forms an interactive pair.

 	\end{itemize}
 \end{exam}

 			%
 	
 	{
 	\begin{defn} \label{Def: derivation and strong derivation on interactive pair}
 		Let $(A,B)$ be an interactive pair. An operator $T_n:(A^\vee)^{\otimes n}\rightarrow A$ is called   \begin{itemize}
 			\item [(i)]  an  {\bf $n$-derivation relative to $B$}, if for all $b_1,b_2\in B$, and $f_1\cdots,f_n\in A^\vee $: 	\begin{eqnarray}\label{Relation:Leibniz identity for homotopy RB}\ \ \ \ \ \ \ \  T_n(f_1\otimes\cdots\otimes f_n){\rhd}(b_1* b_2)=T_n(f_1\otimes \cdots\otimes f_n\btl b_1){\rhd}b_2+\left( T_n(f_1\otimes \cdots\otimes f_n)\rhd b_1\right)* b_2 ;
 			\end{eqnarray}
 			
 			\smallskip
 			
 			\item [(ii)] a {\bf strong    $n$-derivation relative to $B$}, if $T_n$ is an $n$-derivation relative to $B$ and for all  $b_1,b_2\in B$, $g\in B^\vee$, and $f_1,\cdots,f_{n-1}\in A^\vee $, the following identities hold: 	
 			
 			\begin{align}\label{Eq:Cyclic Leibniz identity for homotopy RB at the first componet}
 		 &T_n\big(\kappa(b_1*b_2\otimes f_1)\otimes f_{2}\otimes \cdots\otimes f_{n}\big) \\
 				=&\ (-1)^{|T_n||b_1|}\Big(b_1\btr\big(T_n\big(\kappa(b_2\otimes g)\otimes f_1\otimes \cdots\otimes f_{n-1}\big)\big)\Big) +T_n\big(\kappa\big(b_1\otimes b_2\btr g\big)\otimes f_{1}\otimes \cdots\otimes f_{n-1}\big) \nonumber;
 			\end{align} 

 			\begin{align}\label{Eq:Cyclic Leibniz identity for homotopy RB}
 				 &T_n\big(f_1\otimes\cdots\otimes f_{l-1}\otimes \kappa(b_1*b_2\otimes g)\otimes f_{l}\otimes \cdots\otimes f_{n-1}\big) \\
 			\nonumber	=&T_n(f_1\otimes\cdots\otimes f_{l-1}\btl b_1\otimes \kappa(b_2\otimes g)\otimes f_{l}\otimes \cdots\otimes f_{n-1})  +T_n(f_1\otimes \cdots\otimes f_{l-1}\otimes \kappa(b_1\otimes b_2\btr g)\otimes f_{l}\otimes \cdots\otimes f_{n-1}) ,
 			\end{align} 
 	 for all $1<l\leqslant n$.

 		\end{itemize}  
 	Here $``\btl"$ is the right action of $B$ on $A^\vee$ induced by $``\btr"$ and $\kappa:B\otimes B^\vee\rightarrow A^\vee$  is defined as $\kappa(b\otimes f)(a)=(-1)^{|b|(|f|+|a|)}f(a{\rhd}b)$, for any $b\in B$, $f\in B^\vee$ and $a\in A$.
 	\end{defn}

 	\begin{remark}
 		Given a  interactive pair $(A,B)$, there is an isomorphism:
 		\begin{align*}
 			\iota:	A^{\otimes n}\otimes B&\cong \Hom((A^\vee)^{\otimes n},B) \\
 			a_n\otimes\cdots\otimes a_1\otimes b&\to Q
 		\end{align*}
 		where $Q(f_1\otimes\cdots\otimes f_n)=(-1)^{(\sum_{j=1}^n|f_j|)|b|+(\sum_{j=1}^n|f_j||a_j|)}f_1(a_1)\cdots f_n(a_n)b$, for all $f_1,\ldots,f_n\in A^\vee$.  Since $A$ is a left $B$-module and $B$ is a  right $B$-module, then $A^{\otimes n}\otimes B$ is a   $B$-bimodule.  Therefore, each $n$-derivation $T_n$  relative to $B$ gives rise to a usual derivation   of $B$ into the $B$-bimodule $A^{\otimes n}\otimes B$ : for all $b_1,b_2\in B$
 		\[\iota^{-1}\big(T_n(-\otimes\cdots\otimes-)\rhd (b_1*b_2)\big)=(-1)^{|b_1||T_n|}b_1\btr\iota^{-1}(T_n(-\otimes\cdots\otimes-)\rhd b_2)+\iota^{-1}(T_n(-\otimes\cdots\otimes-)\rhd b_1)*b_2.\]
 		In particular, if one takes $(B,*)$ to be a finite dimensional algebra and $A=\mathrm{End}(B)$, the above construction yields a bijection between the space of $n$-derivation relative to $B$ on $A$ and the space of derivations from $B$ to $A^{\otimes n}\otimes B$.
 	\end{remark}

%
 	
%

 }
 	
 	\begin{prop}\label{Prop: Cyclic Leibniz identity for homotopy RB}
 		Let  $(A,B)$ be an interactive pair with the acting algebra $A$ being  locally finite-dimensional. Let $T_n:(A^\vee)^{\otimes n}\rightarrow A$ be  a cyclic $n$-derivative relative to $B$. Then $T_n$ is  also a strong $n$-derivation relative to $B$.
 		
 	\end{prop}
 	
 	\begin{proof} We will check that $T_n$ satisfies Equations~\eqref{Eq:Cyclic Leibniz identity for homotopy RB at the first componet}\eqref{Eq:Cyclic Leibniz identity for homotopy RB} in Definition~\ref{Def: derivation and strong derivation on interactive pair}.
 		For all  $b_1,b_2,b_3\in B$, $g\in B^\vee$, and $f_1,\cdots,f_{n}\in A^\vee $
 		
 		\begin{align*}	
 			&\langle 	T_n(\kappa(b_1*b_2\otimes g)\otimes f_{1}\otimes \cdots\otimes f_{n-1}),  f_{n}\rangle\\
 			& -(-1)^{|b_1|(\sum\limits_{i=1}^{n}|f_i|+|g|+|b_2|)}\langle   T_n(\kappa(b_2\otimes g)\otimes f_{1}\otimes \cdots\otimes f_{n-1}),  f_{n}\btl b_1\rangle \\
 			&-\langle  T_n(\kappa(b_1\otimes b_2\btr g),f_{1},\cdots,f_{n-1}) , f_{n}\rangle \\	
 			=	&(-1)^{n+(|b_1|+|b_2|+|g|)(\sum\limits_{i=1}^{n}|f_i|)}\langle  	T_n( f_{1}\otimes \cdots\otimes  f_{n}) \rhd(b_1*b_2)  ,  g\rangle \\
 			&-(-1)^{n+(|b_1|+|b_2|+|g|)(\sum\limits_{i=1}^{n}|f_i|)}\langle   T_n( f_{1}\otimes \cdots\otimes  f_{n}\btl b_1)\rhd b_2,   g  \rangle \\
 			&-(-1)^{n+(|b_1|+|b_2|+|g|)(\sum\limits_{i=1}^{n}|f_i|)}\langle  (T_n( f_{1}\otimes \cdots\otimes  f_{n})\rhd b_1 )*b_2,    g\rangle \\
 			=&0
 		\end{align*}
 		
 		Thus we have \begin{align*}	T_n\big(\kappa(b_1*b_2\otimes g)\otimes f_{1}\otimes \cdots\otimes f_{n-1}\big)=&(-1)^{|T_n||b_1|}b_1*(T_n(\kappa(b_2\otimes g)\otimes f_{1}\otimes \cdots\otimes f_{n-1})\\&+T_n(\kappa(b_1\otimes b_2\btr g)\otimes f_{1}\otimes \cdots\otimes f_{n-1}),
 		\end{align*}
 		that is, $T_n$ fulfills Equation~\eqref{Eq:Cyclic Leibniz identity for homotopy RB at the first componet}.
 		
 		Similarly, for any $1< l\leqslant n$, $f_1,\cdots,f_{n}\in A^\vee$, $g\in B^\vee$ and $b_1,b_2\in B$,
 		\begin{align*}
 			&\langle  T_n(f_1\otimes \cdots\otimes f_{l-1}\btl b_1\otimes\kappa(b_2\otimes g)\otimes f_{l}\otimes\cdots\otimes f_{n-1}), f_{n}\rangle \\
 			&+ \langle  T_n(f_1\otimes \cdots\otimes f_{l-1} \otimes \kappa(b_1\otimes b_2\btr g)\otimes f_{l}\otimes \cdots\otimes f_{n-1}),f_{n}\rangle \\
 			&- \langle  T_n(f_1\otimes \cdots\otimes f_{l-1}  \otimes \kappa(b_1*b_2\otimes g)\otimes f_{l}\otimes\cdots\otimes f_{n-1}), f_{n}\rangle \\
 			=&(-1)^\epsilon \langle  T_n(f_{l}\otimes \cdots\otimes f_{n}\otimes f_1\cdots\otimes f_{l-1}\btl b_1), \kappa(b_2\otimes g)\rangle \\
 			&+(-1)^\epsilon \langle  T_n(f_{l}\otimes \cdots\otimes f_{n}\otimes f_1\cdots\otimes f_{l-1} ), \kappa(b_1\otimes b_2\btr g)\rangle \\
 			&-(-1)^\epsilon \langle  T_n(f_{l}\otimes \cdots\otimes f_{n}\otimes f_1\cdots\otimes f_{l-1} ), \kappa(b_1*b_2\otimes   g)\rangle \\
 			=&(-1)^\epsilon\langle  	T_n(f_{l}\otimes \cdots\otimes f_{n}\otimes f_1\cdots\otimes f_{l-1}\btl b_1){\rhd}b_2\\
 			&+\left( T_n(f_{l}\otimes \cdots\otimes f_{n}\otimes f_1\cdots,f_{l-1})\rhd b_1\right)* b_2\\
 			&-T_n(f_{l}\otimes \cdots\otimes f_{n}\otimes f_1\cdots\otimes f_{l-1}){\rhd}(b_1* b_2), g\rangle \\
 			=&0,
 		\end{align*}
 		where $(-1)^\epsilon$ is the Koszul sign determined by  the cyclic permutation. Thus $T_n$ also satisfies Equation~\eqref{Eq:Cyclic Leibniz identity for homotopy RB} for all $1<l\leqslant n$.
 		
 		In conclusion,  $T_n$ is strong $n$-derivative relative to $B$. 
 	\end{proof}

In the remainder of the paper, we mainly work with interactive pairs whose acting algebras are dg homotopy relative Rota-Baxter algebras. Accordingly, we introduce the following concepts.
\begin{defn}\label{Def: homotopy Rota-Baxter interactive pair}
	A {\bf (strong) homotopy Rota-Baxter interactive pair} is an interactive pair \((A, B)\) where the acting algebra \((A, d_A, \cdot)\) is equipped with a dg  homotopy relative Rota-Baxter structure \(\{T_n : (A^\vee)^{\otimes n} \to A\}_{n \geqslant 1}\), such that each \(T_n\) is a (strong) \(n\)-derivation relative to \(B\). 
	
	Moreover, if in a Rota-Baxter interactive pair $(A,B)$, each $T_n$ is cyclic (resp. ultracyclic), then  it will be called a {\bf cyclic (resp. ultracyclic) homotopy Rota-Baxter interactive pair.}
\end{defn}
	
	\medskip
	
\subsection{Constructing pre-Calabi-Yau algebras from cyclic homotopy Rota-Baxter algebras}\ 

We begin by constructing an $A_\infty$-algebra structure on the space $\partial_{-1}B$, where $B$ is the base algebra of a strong homotopy Rota-Baxter interactive pair.

\begin{lem}\label{Lem: A_infinity structures on homotopy RB dg algebras}
	Let $\big((A,B), \{T_n\}_{n \geqslant 1}\big)$ be a strong homotopy Rota-Baxter interactive pair. Define a family of operations $\{m_n\}_{n \geqslant 1}$ on the space $\partial_{-1}B := B \oplus s^{-1}B^\vee$ as follows:
	\begin{itemize}
		\item[\rm(i)] $m_1 := -d_{\partial_{-1}B}$,
		
		\item[\rm(ii)] For all $ b_1, b_2 \in B, f_1, f_2 \in B^\vee$,
		\[
		m_2((b_1, s^{-1}f_1) \otimes (b_2, s^{-1}f_2)) := \left(b_1 * b_2,\ (-1)^{|b_1|}s^{-1}(b_1 \btr f_2) + s^{-1}(f_1 \btl b_2)\right),
		\]
		
		\item[\rm(iii)] For all $b_1, \ldots, b_{n+1} \in B, f_1, \ldots, f_n \in B^\vee$,
		\[m_{2n+1}(b_1 \otimes s^{-1}f_1 \otimes b_2 \otimes \cdots \otimes s^{-1}f_n \otimes b_{n+1})
		:= (-1)^\gamma\, T_n\left(\kappa(b_1 \otimes f_1) \otimes \cdots \otimes \kappa(b_n \otimes f_n)\right) \rhd b_{n+1},\]
		
		\item[\rm(iv)] For all $b_1, \ldots, b_n \in B, f_0, \ldots, f_n \in B^\vee$,
		\[m_{2n+1}(s^{-1}f_0 \otimes b_1 \otimes s^{-1}f_1 \otimes \cdots \otimes b_n \otimes s^{-1}f_n)
		:= (-1)^{|f_0| + \gamma} s^{-1}\left(f_0 \lhd T_n\left(\kappa(b_1 \otimes f_1) \otimes \cdots \otimes \kappa(b_n \otimes f_n)\right)\right),\]
		
		\item[\rm(v)] $m_n$ vanishes in all other cases,
	\end{itemize}
	where
	$$\gamma = \sum\limits_{k=1}^n (n - k + 1)|b_k| + \sum_{k=1}^n (n - k)|f_k|.$$
	Then $\big(\partial_{-1}B, \{m_n\}_{n \geqslant 1}\big)$ forms an $A_\infty$-algebra.
\end{lem}

\begin{proof}
	The proof involves a detailed and technical computation. For clarity and conciseness, we defer the full argument to Appendix~\ref{Appendix: Proof of Lemma: A_infinity structures on homotopy RB dg algebras}.
\end{proof}

\begin{cor}
	Let $(A, A^\vee, \{T_n\}_{n \geqslant1})$ be a  dg homotopy relative Rota-Baxter algebra, and let $B$ be a differential graded left $A$-module. Then the family of operations $\{m_n\}_{n \geqslant 1}$ defined in Lemma~\ref{Lem: A_infinity structures on homotopy RB dg algebras} equips $\partial_{-1}B$ with an $A_\infty$-algebra structure in which $m_2$ is trivial.
\end{cor}

We emphasize that the homotopy Rota-Baxter structure plays a central role in constructing the $A_\infty$-algebra structure described above. Even when $A$ is an ordinary (non-homotopy) Rota-Baxter algebra, the induced $A_\infty$-structure on $\partial_{-1}B$ may still be nontrivial. Consider, for instance, a homotopy Rota-Baxter pair $(A,B)$ in which the acting algebra $A$ is a Rota-Baxter algebra and the base algebra $B$ is a finite-dimensional $A$-module. According to the formulas in Lemma~\ref{Lem: A_infinity structures on homotopy RB dg algebras}, the resulting $A_\infty$-structure $\{m_n\}_{n \geqslant 1}$ on $\partial_{-1}B$ satisfies $m_n = 0$ for all $n \ne 3$, and the only nontrivial operation
$m_3 : (\partial_{-1}B)^{\otimes 3} \to \partial_{-1}B$
is given by:
\begin{align*}
	m_3(b_1 \otimes s^{-1}f_1 \otimes b_2) &= T(\kappa(b_1 \otimes f_1)) \rhd b_2, \\
	m_3(s^{-1}f_1 \otimes b_2 \otimes s^{-1}f_2) &= s^{-1}f_1 \lhd T(\kappa(b_2 \otimes f_2)),
\end{align*}
for $b_1, b_2 \in B$ and $f_1, f_2 \in B^\vee$, and vanishes in all other cases. Notably, the definition of $m_3$ explicitly involves the Rota-Baxter operator $T$.

\smallskip
	
Furthermore, cyclic homotopy Rota-Baxter operators can produce pre-Calabi-Yau algebra structures.
	\begin{thm}\label{Thm: Pre-Calabi-Yau structure on homotopy RB dg algebras}
		Let $(A, B)$ be a homotopy Rota-Baxter interactive pair, where the acting algebra $A$ and the base algebra  $B$   are locally finite-dimensional. Let $\{T_n:(A^\vee)^{\otimes n}\rightarrow A\}_{n\geqslant 1}$ be the homotopy relative Rota-Baxter operator.
		\begin{itemize}
			\item[\rm(i)] If each $T_n$ is cyclic, then $B$ admits a good manageable pre-Calabi-Yau algebra structure.
			\item[\rm(ii)] If each $T_n$ is ultracyclic, then $B$ admits a good manageable special pre-Calabi-Yau algebra structure.
		\end{itemize}
	\end{thm}

	\begin{proof}
		Suppose that each $T_n$ is cyclic and an $n$-derivation relative to $B$. Then, by Proposition~\ref{Prop: Cyclic Leibniz identity for homotopy RB}, each $T_n$ is in fact a strong $n$-derivation relative to $B$. By Lemma~\ref{Lem: A_infinity structures on homotopy RB dg algebras}, this implies that there is an $A_\infty$-algebra structure on $\partial_{-1}B$.
		
		We now verify that this $A_\infty$-algebra is cyclic under the assumption that the homotopy relative Rota-Baxter structure is cyclic. First, note that $m_1$ is cyclic. For $n \geqslant1$, $b_0, \dots, b_n \in B$, and $f_0, \dots, f_n \in B^\vee$, we compute:
		\begin{align*}
			&\zeta_B\big( m_{2n+1}(b_1 \otimes s^{-1}f_1 \otimes \cdots \otimes b_n \otimes s^{-1}f_n \otimes b_0), s^{-1}f_0 \big) \\
			=~& (-1)^\gamma \zeta_B\big( T_n(\kappa(b_1 \otimes f_1) \otimes \cdots \otimes \kappa(b_n \otimes f_n)) \rhd b_0, s^{-1}f_0 \big) \\
			=~& (-1)^{\gamma + (|f_0| - 1)(n - 1 + |b_0| + \sum\limits_{k=1}^n(|b_k| + |f_k|))} f_0\big( T_n(\kappa(b_1 \otimes f_1) \otimes \cdots \otimes \kappa(b_n \otimes f_n)) \rhd b_0 \big) \\
			=~& (-1)^{2n - 1 + (|f_0| - 1)(n + |b_0| + \sum\limits_{k=1}^n(|b_k| + |f_k|))} \zeta_B\big( m_{2n+1}(s^{-1}f_0 \otimes b_1 \otimes \cdots \otimes b_n \otimes s^{-1}f_n), b_0 \big).
		\end{align*}
		
		By Proposition~\ref{Prop: From relative cyclic RB to absolute cyclic RB}, the induced operators $\{\overline{T}_n\}_{n \geqslant1}$ on $\partial_0 A$ form a cyclic homotopy Rota-Baxter operator. Thus,
		\begin{align*}
			&\zeta_B\big( m_{2n+1}(s^{-1}f_0 \otimes b_1 \otimes \cdots \otimes b_n \otimes s^{-1}f_n), b_0 \big) \\
			=~& (-1)^{|f_0| + \gamma} f_0\big( T_n(\kappa(b_1 \otimes f_1) \otimes \cdots \otimes \kappa(b_n \otimes f_n)) \rhd b_0 \big) \\
			=~& (-1)^{|f_0| + \gamma + |f_0|(n - 1 + |b_0| + \sum\limits_{k=1}^n(|b_k| + |f_k|))} \left\langle T_n(\kappa(b_1 \otimes f_1) \otimes \cdots \otimes \kappa(b_n \otimes f_n)), \kappa(b_0 \otimes f_0) \right\rangle \\
			=~& (-1)^{|f_0| + \gamma + |f_0|(n - 1 + |b_0| + \sum\limits_{k=1}^n(|b_k| + |f_k|)) + n + (|b_0| + |f_0|)\sum\limits_{k=1}^n(|b_k| + |f_k|)} \\
			& \quad \left\langle T_n(\kappa(b_0 \otimes f_0) \otimes \cdots \otimes \kappa(b_{n-1} \otimes f_{n-1})), \kappa(b_n \otimes f_n) \right\rangle \\
			=~& (-1)^{2n - 1 + |b_0|(n + 1 + |f_0| + \sum\limits_{k=1}^n(|b_k| + |f_k|)) + \sum\limits_{k=0}^{n-1}(n - k + 1)|b_k| + \sum\limits_{k=0}^{n-1}(n - k)|f_k|} \\
			& \quad \zeta_A\big( T_n(\kappa(b_0 \otimes f_0) \otimes \cdots \otimes \kappa(b_{n-1} \otimes f_{n-1})) \rhd b_n, s^{-1}f_n \big) \\
			=~& (-1)^{2n - 1 + |b_0|(n + 1 + |f_0| + \sum\limits_{k=1}^n(|b_k| + |f_k|))} \zeta_B\big( m_{2n+1}(b_0 \otimes s^{-1}f_0 \otimes \cdots \otimes b_n), s^{-1}f_n \big).
		\end{align*}
		
		Hence, $\partial_{-1}B$ is a $(-1)$-cyclic $A_\infty$-algebra containing $B$ as an $A_\infty$-subalgebra; that is, $B$ is a pre-Calabi-Yau algebra. By the construction in Lemma~\ref{Lem: A_infinity structures on homotopy RB dg algebras}, this pre-Calabi-Yau algebra is good and manageable.
		
		Now assume further that each $T_n$ is skew-symmetric. For each $n \geqslant1$, $b_1, \dots, b_{n+1} \in B$, $f_1, \dots, f_{n+1} \in B^\vee$, and $\sigma \in \mathfrak{S}_{n}$, we have:
		\begin{align*}
			&\zeta_B\big( m_{2n+1}(b_1 \otimes s^{-1}f_1 \otimes \cdots \otimes b_n \otimes s^{-1}f_n \otimes b_{n+1}), s^{-1}f_{n+1} \big) \\
			=~& (-1)^{\gamma + (|f_{n+1}| - 1)(n - 1 + |b_{n+1}| + \sum\limits_{k=1}^n(|b_k| + |f_k|))} f_{n+1}\big( T_n(\kappa(b_1 \otimes f_1) \otimes \cdots \otimes \kappa(b_n \otimes f_n)) \rhd b_{n+1} \big) \\
			=~& \chi(\sigma; b_1 \otimes f_1 \otimes \cdots \otimes b_n \otimes f_n) f_{n+1}\big( T_n(\kappa(b_{\sigma(1)} \otimes f_{\sigma(1)}) \otimes \cdots \otimes \kappa(b_{\sigma(n)} \otimes f_{\sigma(n)})) \rhd b_{n+1} \big) \\
			=~& \varepsilon(\sigma; b_1 \otimes s^{-1}f_1 \otimes \cdots \otimes b_n \otimes s^{-1}f_n)  \zeta_B\big( m_{2n+1}(b_{\sigma(1)} \otimes s^{-1}f_{\sigma(1)} \otimes \cdots \otimes b_{\sigma(n)} \otimes s^{-1}f_{\sigma(n)} \otimes b_{n+1}), s^{-1}f_{n+1} \big).
		\end{align*}
		
		Similarly,
		\begin{align*}
			&\zeta_B\big( m_{2n+1}(s^{-1}f_1 \otimes b_1 \otimes \cdots \otimes s^{-1}f_n \otimes b_n \otimes s^{-1}f_{n+1}), b_{n+1} \big) \\
			=~& \varepsilon(\sigma; s^{-1}f_1 \otimes b_1 \otimes \cdots \otimes s^{-1}f_n \otimes b_n)  \zeta_B\big( m_{2n+1}(s^{-1}f_{\sigma(1)} \otimes b_{\sigma(1)} \otimes \cdots \otimes s^{-1}f_{\sigma(n)} \otimes b_{\sigma(n)} \otimes s^{-1}f_{n+1}), b_{n+1} \big).
		\end{align*}
		
	We already know that $m_{2n+1}$ is cyclic. Moreover, by the skew-symmetry of $T_n$, we conclude that $m_{2n+1}$ is ultracyclic.	Thus, if $\{T_n\}_{n \geqslant1}$ is ultracyclic, then $B$ is a special pre-Calabi-Yau algebra.
	\end{proof}

\begin{remark}
	In fact, the assumption that the acting algebra $A$ is locally finite-dimensional in Theorem~\ref{Thm: Pre-Calabi-Yau structure on homotopy RB dg algebras} is not essential. The theorem remains valid even when $A$ is not locally finite-dimensional, and in such cases, the proof can still be carried out through direct computation.
\end{remark}

\begin{cor}
	Let $(A, d_A, \cdot, \{T_n\}_{n \geqslant 1})$ be a  locally finite-dimensional cyclic dg homotopy Rota-Baxter algebra, and let $B$ be a locally finite-dimensional dg module over the dg algebra $(A, d_A, \cdot)$. Then $B$ admits a fine pre-Calabi-Yau algebra structure.
\end{cor}

\begin{proof}
	Since $A$ is a locally finite-dimensional cyclic dg homotopy Rota-Baxter algebra, it is in particular a dg homotopy relative Rota-Baxter algebra. Let $B$ be a dg $A$-module. According to Example~\ref{Exm: interactive pair}(2), the pair $(A, B)$ always forms an interactive pair and is clearly homotopy Rota-Baxter compatible. The result then follows directly from Theorem~\ref{Thm: Pre-Calabi-Yau structure on homotopy RB dg algebras}.
\end{proof}

	\begin{cor}\label{Coro: Pre-Calabi-Yau structure  on graded space induced by  homotopy RB on endomorphism algebra}
		 Let $B$ be a finite dimensional graded space, $A$ the graded algebra $\mathrm{End}(B)$ with the composition being multiplication.   
		 Then   the following four maps given by Lemma~\ref{Lem: A_infinity structures on homotopy RB dg algebras} are bijections:
		 \begin{align*}
		 &\left\{\begin{array}{l}
		 	\text {pairs ($d_B$, $\{T_n\}_{\geqslant 1}$) where $d_B$ is a differential on }\\
		 	\text{$B$ and $\{T_n: (A^\vee)^{\otimes n}\rightarrow A\}_{n\geqslant 1}$ is a cyclic }\\
		 	\text{ homotopy relative Rota-Baxter operator }
		 \end{array}\!\right\} \rightarrow\left\{\begin{array}{l}
		 	\text { fine  pre-Calabi-Yau algebra  }  \\
		 	\text { structures on $B$ }
		 \end{array}\ \ \ \ \ \ \ \ \ \ \ \ \  \right\},\\
		 &\left\{\begin{array}{l}
		 	\text {triples $(d_B, m, \{T_n\}_{n\geqslant 1})$ where $(B,d_B,m)$ is a dg}\\
		 	\text{algebra and $(A,B,\{T_n\}_{n\geqslant 1})$ forms a cyclic }\\
		 	\text{homotopy Rota-Baxter interactive pair }
		 \end{array}\right\} \rightarrow\left\{\begin{array}{l}
		 	\text { good manageable  pre-Calabi-Yau  }  \\
		 	\text {algebra structures on $B$ }
		 \end{array}\ \ \ \ \  \right\},\\
		 &\left\{\begin{array}{l}
		 \text { pairs $(d_B,\{T_n\}_{n\geqslant 1})$ where $d_B$ is a differential on }\\\
		 \text{$B$ and $\{T_n:(A^\vee)^{\otimes n}\rightarrow A\}_{n\geqslant 1}$ is an ultracyclic}\\
		 \text{ homotopy Rota-Baxter operator}
		 \end{array}\!\!\!\right\} \rightarrow\left\{\begin{array}{l}
		 	\text {fine special pre-Calabi-Yau algebra}  \\
		 	\text {  structures on $B$ }
		 \end{array}\ \ \ \  \right\},\\
		 &\left\{\begin{array}{l}
		 	\text{triples $(d_B, m,\{T_n\}_{\geqslant 1})$ where $(B,d_B,m)$ is a dg }
		 	\\\text{algebra and $\{T_n\}_{n\geqslant 1}$ makes $(A,B)$ into  an }\\
		 	\text{ultracyclic homotopy Rota-Baxter interactive}\\
		 	\text{ pair}
		 \end{array}\ \right\} \rightarrow\left\{\!\!\begin{array}{l}
		 	\text { good manageable special  }  \\
		 	\text { pre-Calabi-Yau  algebra structures on $B$  }
		 \end{array}\!\!\!\right\}, 		 
		 \end{align*}
		 where $A$ is always endowed with the induced differential by $d_B$.
		 \end{cor}
	\begin{proof} 
		Each good map $m_{2n+1}$ can be uniquely determined by an operator $ \widetilde{T}_n:(B\otimes B^\vee)^{\otimes n}\to B\otimes B^\vee$. Since $\kappa:B\otimes B^\vee\rightarrow \End(B)^\vee$ is an isomorphism, the  maps are  bijective. 
	\end{proof}

 By Theorem~\ref{Thm: Pre-Calabi-Yau structure on homotopy RB dg algebras} and the cyclic completion for homotopy Rota-Baxter algebras Proposition~\ref{Prop:From homotopy Rota-Baxter algebras to cyclic homotopy Rota-Baxter algerbas}, we have the following result.
	\begin{prop}
			 Let $(A,\{T_n\}_{n\geqslant1})$  be a  locally finite-dimensional dg homotopy Rota-Baxter algebra. Then  there is  a fine pre-Calabi-Yau structure $\{m_n\}_{n\geqslant 1}$ on any locally finite-dimensional  left dg $\partial_0A$-module $M$. Moreover, $\{T_n\}_{n\geqslant1}$ is skew-symmetric, the pre-Calabi-Yau algebra structure on $M$ is fine and special. 
	\end{prop}
	
%

\bigskip

	\section{Homotopy Rota-Baxter algebras and double Poisson structures }\label{Section:Homotopy   Rota-Baxter algebras and double  Poisson structures}


In Section~\ref{Section: Pre-Calabi-Yau structures arising from cyclic homotopy Rota-Baxter algebras}, we constructed a good manageable (resp. good manageable special) pre-Calabi-Yau algebra on the base algebra of a homotopy Rota-Baxter interactive pair endowed with a cyclic (resp. an ultracyclic) homotopy Rota-Baxter operator $\{T_n\}_{n \geqslant1}$. In \cite{FH21}, Fernández and Herscovich established an equivalence between good manageable special pre-Calabi-Yau algebras and homotopy double Poisson algebras. In the present section, we combine these results to give a direct construction of a homotopy double Poisson algebra from a homotopy Rota-Baxter structure. Specifically, we show that the base algebra of a ultracyclic (resp. cyclic) homotopy Rota-Baxter interactive pair naturally inherits a (resp. cyclic) homotopy double Poisson structure. Moreover, we observe that any module over an ultracyclic homotopy relative Rota-Baxter algebra carries a homotopy double Lie structure, from which it follows that the symmetric algebra on such a module acquires the structure of a homotopy Poisson algebra. As an application, we establish an equivalence between skew-symmetric solutions of the associative Yang-Baxter-infinity equations, ultracyclic homotopy Rota-Baxter algebra structures, fine special pre-Calabi-Yau algebras, and homotopy double Lie algebras.

	\subsection{Double Poisson structures arising from homotopy Rota-Baxter structures}\label{Subection:Homotopy   Rota-Baxter algebras and Poisson structures}\ 
	
	Let's recall some basics on homotopy Poisson algebras and homotopy double Poisson algebras following  \cite{Sched09,FH21}.

	\begin{defn}\label{Def:HDLA}
		A  {\bf cyclic homotopy double Lie algebra} ( also called {\bf a cyclic double $L_\infty$-algebra}) is a   graded space $V=\oplus_{n \in \mathbb{Z}} V_n$ equipped with a family of homogeneous maps $\lbb-,\cdots,-\rbb_{n}: V^{\otimes n} \rightarrow V^{\otimes n}$ with $|\lbb-,\cdots,-\rbb_{n}|=n-2$ satisfying the following conditions  for all $n \geqslant1$,   
		\begin{itemize}
			\item[(i)] Cyclic-symmetry: For all elements $\sigma\in \mathfrak{C}_n$(the cyclic group of $n$ elements) 
			$$ \sigma \circ \lbb-,\cdots,-\rbb_{n} \circ \sigma^{-1}=\sgn(\sigma)\lbb-,\cdots,-\rbb_{n} ;$$
			\item [(ii)] Double Jacobi$ _\infty$ identity:
			\begin{eqnarray}\label{DJac infinty}
				\sum\limits_{i+j=n+1} (-1)^{(j-1)i} \sum\limits_{\sigma \in \mathfrak{C}_{n}} \sgn (\sigma) \sigma\circ  \lbb-,\cdots,-,\lbb-,\cdots,-\rbb_{i}\rbb_{L,j}     \circ\sigma^{-1}=0,
			\end{eqnarray}
			where
			$$
			\lbb-,\cdots,-,\lbb-,\cdots,-\rbb_{i+1}\rbb_{L,j+1}=\left(\lbb-,\cdots,-\rbb_{ j+1 } \otimes \id_ {V} ^ { \otimes i  } \right)  \circ \left(\id_V^{\otimes j} \otimes\lbb-,\cdots,-\rbb_{i+1} \right) .
			$$
		\end{itemize}
	If, in addition,  each map $\lbb-,\cdots,-\rbb_{n} $  is skew-symmetric, meaning that for all $\sigma\in \mathfrak{S}_n$, $$\sigma\circ  \lbb-,\cdots,-\rbb_{n} \circ \sigma^{-1}=\sgn(\sigma)\lbb-,\cdots,-\rbb_{n}  \text{ , for all }  \sigma \in \mathfrak{S}_{n},$$  then  $(V,\lbb-,\cdots,-\rbb_{n})$is called {\bf double $L_\infty$-algebra} (also known as {\bf homotopy double Lie algebra}).
		
	\end{defn}

	The following lemma offers an alternative characterization of a homotopy double Lie algebra, which will be used later.
	\begin{lem}\label{Lemma:def homotopy double Lie2}
		Let $\{\lbb-,\cdots,-\rbb_{n}: V^{\ot n}\rightarrow V^{\ot n}\}_{n\geqslant1}$ be a family of  operations on a graded space $V=\oplus_{n \in \mathbb{Z}} V^n$. For each $k\geqslant 1$, define the opposite bracket $\lbb-,\cdots,-\rbb^{\rm{op}}_{k}:=\sigma_{k}\circ\lbb-,\cdots,-\rbb_{k}\circ\sigma^{- 1}_k  $, where $\sigma_k\in \mathfrak{S}_k$ is the order-reversing permutation
		$$\sigma_{k}=\left( \begin{array}{cccc}
			1\ &2\ &\cdots &k\\
			k\ &k-1\ &\cdots &1\\
		\end{array}\right)\in \mathfrak{S}_k .$$ 
		
		 Then the family $\{\lbb-,\cdots,-\rbb_{n}\}_{n\geqslant 1}$ satisfies the double Jacobi$ _\infty$ identity if and only the opposite  operations $\{\lbb-,\cdots,-\rbb^{\rm{op}}_{n+1}\}_{n\geqslant 0}$ fulfill the following identities:
		\begin{eqnarray}\label{DJac infinty2}
			\sum\limits_{i+j=n+1} (-1)^{i(j-1)} \sum\limits_{\sigma \in \mathfrak{C}_{n}} \sgn (\sigma) \sigma \circ \left( \lbb\lbb-,\cdots,-\rbb^{\rm{op}}_{i},-,\cdots,-\rbb^{\rm{op}}_{R,j}\right)     \circ\sigma^{-1}=0,
		\end{eqnarray}
		where the right-nested composite is defined by
		$$ \lbb\lbb-,\cdots,-\rbb^{\rm{op}}_{i+1},-,\cdots,-\rbb^{\rm{op}}_{R,j+1}= \left( \id_ {V} ^ { \otimes i  }   \otimes \lbb-,\cdots,-\rbb^{\rm{op}}_{ j+1 }  \right)  \circ \left( \lbb-,\cdots,-\rbb^{\rm{op}}_{i+1}\otimes  \id_V^{\otimes j} \right)  .$$
	\end{lem}

	\begin{proof} The claim follows by applying the conjugation \( \sigma_n \circ (\text{Equation}~\eqref{DJac infinty}) \circ \sigma_n^{-1} \), which transforms the original double $\mbox{Jacobi}_\infty$ identity into Equation~\eqref{DJac infinty2}.
	\end{proof}

\begin{defn}\label{Definition: homotopy double Poisson algebra}
	\begin{itemize}
		\item[(i)] A {\bf cyclic homotopy double Poisson algebra} is a graded vector space $A$ equipped with both an associative algebra structure and a cyclic double $L_\infty$-algebra structure, satisfying the \emph{double Leibniz$_\infty$ rule}: for all $n \geq 0$ and homogeneous elements $a_1, \dots, a_{n-1}, a_n', a_n'' \in A$,
		\begin{align*}
			\lbb a_1, \dots, a_n, a_{n+1}' a_{n+1}'' \rbb_n 
			&= \lbb a_1, \dots, a_n', \rbb_n \cdot a_{n+1}'' \\
			&\quad + (-1)^{|a_{n+1}'|(n - 2 + \sum\limits_{k=1}^{n} |a_k|)} \, a_{n+1}' \cdot \lbb a_1, \dots, a_n'' \rbb_n,
		\end{align*}
		where multiplication by $a_{n+1}''$ and $a_{n+1}'$ is understood to act on the rightmost and leftmost components of the tensor product, respectively.
		
		\item[(ii)] A {\bf double Poisson$_\infty$ algebra} (also called a {\bf homotopy double Poisson algebra}) is a graded algebra $A$ equipped with a double $L_\infty$-algebra structure that satisfies the double Leibniz$_\infty$ rule.
	\end{itemize}
\end{defn}

Next, we recall the following result of Fernández and Herscovich \cite{FH21}, which establishes a connection between ultracyclic pre-Calabi-Yau algebras and homotopy double Poisson algebras.

	\begin{thm}\label{Theorem: Bijection between ultracyclic pre-Calabi-Yau and homotopy double Poisson}\cite[Theorem~6.3]{FH21} \
	 Let $A=\oplus_{n \in \mathbb{Z}} A^n$ be a  finite dimensional  graded space.  For a good manageable special pre-Calabi-Yau structure $\{m_n\}_{n\geqslant1}$ on $A$, define a family of maps $\{\lbb-,\cdots,-\rbb_{n}: A^{\ot n}\rightarrow A^{\ot n}\}_{n\geqslant1}$ by
	 
	 \begin{eqnarray}\label{Eq:pre-CY and homotopy double Poisson}
	 	 \left(f_1 \otimes \cdots \otimes f_n\right)\left(\lbb a_1, \cdots, a_n\rbb_n\right)=s_{f_1, \cdots, f_n}^{a_1, \cdots, a_n} \zeta_A\left(m_{2 n-1}\left(a_n, s^{-1} f_n, \cdots, a_2, s^{-1} f_2, a_1\right), s^{-1} f_1\right)
	 \end{eqnarray}
		for all homogeneous elements $a_1, \cdots, a_n \in A$ and $f_1, \cdots, f_n \in A^{\vee}$, where
		$$
		\begin{aligned}
		 s_{f_1, \cdots, f_n}^{a_1, \cdots, a_n}=	&(-1)^{|a_n||f_1|+(n+1)\left(|a_n|+|f_1|\right)+\sum\limits_{j=1}^n(n-j)|a_j|+\sum\limits_{j=1}^n(j-1)|f_j|+\sum\limits_{1 \leqslant i<  j<  n}|a_i||a_j|+\sum\limits_{1<  i<  j \leqslant n}|f_i||f_j|+\sum\limits_{1<  i \leqslant j<  n}|f_i||a_j|}. 
		\end{aligned}
		$$
	The family of maps $\{\lbb-,\cdots,-\rbb_{n}\}_{n\geqslant1}$, together with the dg algebra structure on $A$,  defines  a homotopy double  Poisson algebra structure on the graded space $A$.
	
	 Moreover, the assignment
		$$
		\left\{\begin{array}{l}
			\text { good manageable special pre-Calabi-Yau} \\
			\text {  algebra structures }\{m_n\}_{n\geqslant1} \text {  on  } A
		\end{array}\right\} \rightarrow\left\{\begin{array}{l}
			\text {homotopy  double Poisson algebra } \\
			\text { structures }   \{\lbb-,\cdots,-\rbb_{n}\}_{n\geqslant1} \text{ on }A 
		\end{array}\right\}
		$$
		
		defined by \eqref{Eq:pre-CY and homotopy double Poisson} is a bijection. 
	\end{thm}
	
{In fact, when Fernández and Herscovich prove Theorem~\ref{Theorem: Bijection between ultracyclic pre-Calabi-Yau and homotopy double Poisson} in \cite{FH21}, the assumption that the pre-Calabi-Yau structure is ultracyclic is used solely to guarantee that  all the operations $\{\lbb -,\dots,- \rbb_n\}_{n\geqslant 1}$ are skew-symmetric. In verifying that the family \(\{\lbb-,\cdots,-\rbb_n\}_{n\geqslant 1}\) satisfies the double $\text{Leibniz}_\infty$ rule and the double $\text{Jacobi}_\infty$ identities, only the cyclicity of the pre-Calabi-Yau structure is required. Therefore, without assuming that the pre-Calabi-Yau algebra is special, the bijection in the above theorem extends to a correspondence between the class of good manageable pre-Calabi-Yau structures and the class of cyclic homotopy double Poisson algebra structures. Thus we have 

\begin{thm}
	\label{Theorem: Bijection between pre-Calabi-Yau and cyclic homotopy double Poisson}\
	Let $A=\oplus_{n \in \mathbb{Z}} A^n$ be a  finite dimensional  graded space.  Given a good manageable  pre-Calabi-Yau structure $\{m_n\}_{n\geqslant1}$ on $A$, define a family of maps $\{\lbb-,\cdots,-\rbb_{n}: A^{\ot n}\rightarrow A^{\ot n}\}_{n\geqslant1}$  as in \eqref{Eq:pre-CY and homotopy double Poisson}.
	Then the family of maps $\{\lbb-,\cdots,-\rbb_{n}\}_{n\geqslant1}$, together with the dg algebra structure on $A$,  defines  a cyclic  homotopy double  Poisson algebra structure on the graded space $A$.
	
	Moreover, the assignment
	$$
	\left\{\begin{array}{l}
		\text { good manageable pre-Calabi-Yau} \\
		\text {  algebra structures }\{m_n\}_{n\geqslant1} \text {  on  } A
	\end{array}\right\} \rightarrow\left\{\begin{array}{l}
		\text {cyclic homotopy  double Poisson algebra } \\
		\text { structures }   \{\lbb-,\cdots,-\rbb_{n}\}_{n\geqslant1} \text{ on }A 
	\end{array}\right\}
	$$
	
	defined by \eqref{Eq:pre-CY and homotopy double Poisson} is a bijection. 
	\end{thm}

}

As a direct consequence of Theorem~\ref{Theorem: Bijection between ultracyclic pre-Calabi-Yau and homotopy double Poisson} and Theorem~\ref{Theorem: Bijection between pre-Calabi-Yau and cyclic homotopy double Poisson}  , we have the following result:
\begin{cor} \label{Coro: Correspongdence between some particular pre-Calabi-Yau and homotopy double Poisson}
The following three maps are bijections via \eqref{Eq:pre-CY and homotopy double Poisson}:
	\begin{align*}
&\left\{\begin{array}{l}
	\text { fine  pre-Calabi-Yau algebra } \\
	\text {  structures }\{m_n\}_{n\geqslant1} \text {  on  $ A$} 
\end{array}\ \ \ \ \ \ \ \ \right\} \rightarrow\left\{\begin{array}{l}
	\text {cyclic homotopy  double Lie algebra} \\
	\text {structures }   \{\lbb-,\cdots,-\rbb_{n}\}_{n\geqslant1} \text{ on }A 
\end{array}\ \ \ \ \ \ \ \ \right\},
\\
&\left\{\begin{array}{l}
	\text { good manageable pre-Calabi-Yau} \\
	\text {  algebra structures }\{m_n\}_{n\geqslant1} \text {  on  } A
\end{array}\right\} \rightarrow\left\{\begin{array}{l}
	\text {cyclic homotopy  double Poisson algebra} \\
	\text {structures }   \{\lbb-,\cdots,-\rbb_{n}\}_{n\geqslant1} \text{ on }A 
\end{array}\ \right\},\\
&\left\{\begin{array}{l}
	\text { fine special pre-Calabi-Yau} \\
	\text {  algebra structures }\{m_n\}_{n\geqslant1} \text {  on  } A
\end{array}\ \ \ \ \right\} \rightarrow\left\{\begin{array}{l}
	\text { homotopy  double Lie algebra} \\
	\text {    $\{\lbb-,\cdots,-\rbb_{n}\}_{n\geqslant1}$ on $A$} 
\end{array}\ \ \ \ \ \ \ \ \ \ \ \ \ \ \ \ \ \ ~\right\}.
\end{align*}
\end{cor}

In Theorem~\ref{Thm: Pre-Calabi-Yau structure on homotopy RB dg algebras}, we constructed pre-Calabi-Yau structures from homotopy Rota-Baxter algebras. By combining this construction with Theorem~\ref{Theorem: Bijection between ultracyclic pre-Calabi-Yau and homotopy double Poisson} and Theorem~\ref{Theorem: Bijection between pre-Calabi-Yau and cyclic homotopy double Poisson}, we obtain the following result, which provides a method for constructing homotopy double Poisson algebras from homotopy Rota-Baxter structures.

\begin{thm} \label{Thm:From RB infinity alegbra to double Poisson infinity}
	Let $(A, B)$ be a homotopy Rota-Baxter interactive pair, where the acting algebra $A$ is finite-dimensional and the base algebra $B$ is locally finite-dimensional. Let $\{T_n : (A^\vee)^{\otimes n} \to A\}_{n \geqslant 1}$ be a relative differential graded homotopy Rota-Baxter operator on $A$.
	
	Define a sequence of maps $\{\lbb -, \ldots, - \rbb_n : B^{\otimes n} \to B^{\otimes n} \}_{n \geqslant 1}$ by setting $\lbb - \rbb_1 = d_B$, and for all $n \geqslant 1$,
	\begin{equation} \label{Eq:Homotopy double Poisson construction}
		\lbb -, \ldots, - \rbb_{n+1} := \Psi^n(\id_{A^{\otimes n}}),
	\end{equation}
	where the map $\Psi^n$ is the composition:
	\[
	\Psi^n : \End(A^{\otimes n}) \cong A^{\otimes n} \otimes (A^\vee)^{\otimes n} 
	\xrightarrow{\id^{\otimes n} \otimes T_n} A^{\otimes (n+1)} 
	\xrightarrow{\Phi^{\otimes (n+1)}} \End(B)^{\otimes (n+1)} 
	\to \End(B^{\otimes (n+1)}),
	\]
	and $\Phi : A \to \End(B)$ denotes the left $A$-action on $B$, i.e., $\Phi(a)(b) := a \rhd b$.
	
	Then, 
	\begin{itemize}
		\item[\rm (i)]If each $T_n$ is  cyclic, the collection $\{\lbb -, \ldots, - \rbb_n\}_{n \geqslant 1}$ defines a cyclic homotopy double Poisson algebra structure on $B$.
		\item[\rm (ii)] If each $T_n$ is ultracyclic, the collection $\{\lbb -, \ldots, - \rbb_n\}_{n \geqslant 1}$ defines a homotopy double Poisson algebra structure on $B$.
	\end{itemize}
\end{thm}

%
%
%
%

	\begin{proof} 
		Let $\{e_i\}_{i\in I}$ be a homogeneous basis of $A$ and $\{e^i\}_{i\in I}$ be the corresponding dual basis. Then  $ \id_{A^{\ot n}}\in \End(A^{\ot n})$  corresponds to the element $ \sum\limits_{ i_1,\cdots,i_n}   e_{i_1} \otimes  \dots\otimes e_{i_n}\otimes e^{ i_n}\otimes \dots\otimes e^{i_1} \in A^{\ot n} \ot (A^\vee)^{\ot n}$. Thus, we can write $$\lbb-,\cdots,-\rbb_{n+1}=\Phi^{\ot n+1}\Big( \sum\limits_{ i_1,\cdots,i_n} (-1)^{(n-1)(\sum\limits_{k=1}^n|e_{i_k}|)}e_{i_1} \otimes \dots\otimes e_{i_n}\otimes T_n (e^{ i_n}\otimes \dots\otimes e^{i_1})\Big) .$$
		
	It remains to verify that the image of the $A_\infty$-structure $\{m_n\}_{n \geqslant 1}$ under the construction given in \eqref{Eq:pre-CY and homotopy double Poisson}, as described in Theorems~\ref{Theorem: Bijection between ultracyclic pre-Calabi-Yau and homotopy double Poisson} and~\ref{Theorem: Bijection between pre-Calabi-Yau and cyclic homotopy double Poisson}, coincides with the family \(\{\lbb -, \ldots, - \rbb_n\}_{n \geqslant 1}\) defined by \eqref{Eq:Homotopy double Poisson construction}.
		Let $b_1,\cdots,b_n\in B$ and $f_1,\cdots,f_n\in B^\vee$
		\begin{align*}
			  &(f_1 \otimes \cdots \otimes f_{n+1})(\lbb b_1, \cdots, b_{n+1}\rbb_{n+1}) \\
			  =&(f_1 \otimes \cdots \otimes f_{n+1})\Phi^{\ot n+1}\Big( \sum\limits_{ i_1,\cdots,i_n} (-1)^{(n-1)(\sum\limits_{k=1}^n|e_{i_k}|)}e_{i_1} \otimes \dots\otimes e_{i_n}\otimes T_n (e^{ i_n}\otimes \dots\otimes e^{i_1})\Big)( b_1\otimes\cdots\otimes b_{n+1}) \\
			  =&  \sum\limits_{ i_1,\cdots,i_n}(-1)^{(n-1)(\sum\limits_{k=1}^n(|e_{i_k}|+|b_k|))+\sum\limits_{1\leqslant i<  j\leqslant n+1} |b_i||f_{j}|+\sum\limits_{1\leqslant s<  k\leqslant j\leqslant n} (|b_j|+|f_{s}|)|e_{i_k}|+|f_{n+1}|(\sum\limits_{k=1}^n|e_{i_k}|)}\\
			  & \ \ \ \ \ \ f_1(e_{i_1}\rhd b_1)\cdots f_n(e_{i_n}\rhd b_n)f_{n+1}\big(T_n (e^{ i_n}\otimes \dots\otimes e^{i_1})\rhd b_{n+1} \big) \\
			  =&  (-1)^{(n-1)(\sum\limits_{k=1}^n|f_k|)+\sum\limits_{1\leqslant i<  j\leqslant n+1} |b_i||f_{j}|+\sum\limits_{1\leqslant s<  k\leqslant j\leqslant n} (|b_j|+|f_{s}|)(|b_k|+|f_k|)+|f_{n+1}|(\sum\limits_{k=1}^n(|b_k|+|f_k|))} f_{n+1}\Big( T_n\big(\kappa(b_n\otimes f_n),\cdots, \kappa(b_1\otimes f_1)\big)\rhd b_{n+1}\Big) \\
			  =&  (-1)^{(n-1)(\sum\limits_{k=1}^{n+1}|f_k|)+\sum_{1\leqslant i<  j\leqslant n+1} |b_i||f_{j}|+\sum\limits_{1\leqslant s<  k\leqslant j\leqslant n} (|b_j|+|f_{s}|)(|b_k|+|f_k|)+|f_{n+1}||b_{n+1}|}\langle  T_n\big(\kappa(b_n\otimes f_n),\cdots, \kappa(b_1\otimes f_1)\big),\kappa (b_{n+1}\otimes f_{n+1}) \rangle  \\
			    =&  (-1)^{(n-1)(\sum\limits_{k=1}^{n+1}|f_k|)+\sum\limits_{1\leqslant i<  j\leqslant n+1} |b_i||f_{j}|+\sum\limits_{1\leqslant s<  k\leqslant j\leqslant n} (|b_j|+|f_{s}|)(|b_k|+|f_k|)+|f_{n+1}||b_{n+1}|+n+(|b_{n+1}|+|f_{n+1}|)(\sum\limits_{k=1}^n(|b_k|+|f_k|))}\\
			  & \ \ \ \ \ \ \  \langle  T_n\big(\kappa(b_{n+1}\otimes f_{n+1}),\cdots, \kappa(b_2\otimes f_2)\big),\kappa (b_{1}\otimes f_{1}) \rangle  \\
			   =&(-1)^ns_{f_1, \cdots, f_{n+1}}^{b_1, \cdots, b_{n+1}} \zeta_B(m_{2 n+1}(b_{n+1}\otimes s^{-1} f_{n+1}\otimes \cdots\otimes b_2\otimes s^{-1} f_2\otimes b_1), s^{-1} f_1),
		\end{align*}
	where $\{m_n\}_{n\geqslant  1}$ is defined as Lemma~\ref{Lem: A_infinity structures on homotopy RB dg algebras}. Thus, the image of the operation $m_{2n-1}$ under the construction given in \eqref{Eq:pre-CY and homotopy double Poisson} coincides with \(\lbb -,\ldots,- \rbb_n\) up to a sign $(-1)^n$, as defined in \eqref{Eq:Homotopy double Poisson construction}, for all $n \geqslant 1$. By Theorem~\ref{Thm: Pre-Calabi-Yau structure on homotopy RB dg algebras}, if each $T_n$ is cyclic (resp. ultracyclic), the collection $\{m_n\}_{n \geqslant 1}$ defines a cyclic (resp. ultracyclic) pre-Calabi–Yau algebra structure on $B$. Consequently, by Theorems~\ref{Theorem: Bijection between ultracyclic pre-Calabi-Yau and homotopy double Poisson} and~\ref{Theorem: Bijection between pre-Calabi-Yau and cyclic homotopy double Poisson}, the family \(\{\lbb -, \ldots, - \rbb_n\}_{n \geqslant 1}\) endows $B$ with a homotopy double Poisson algebra structure (resp. cyclic homotopy double Poisson algebra structure).
	\end{proof}
	
As a corollary, we have the following result:

\begin{cor}
	Let $A$ be a finite-dimensional dg algebra, and let $B$ be a locally finite-dimensional dg left $A$-module. Suppose $\{T_n : (A^\vee)^{\otimes n} \to A\}_{n \geqslant 1}$ is a homotopy relative Rota-Baxter operator on $A$. 
	
	If each $T_n$ is ultracyclic (resp. cyclic), then the family of operations $\{\lbb -, \ldots, - \rbb_n\}_{n \geqslant 1}$ defined in Theorem~\ref{Thm:From RB infinity alegbra to double Poisson infinity} endows $B$ with  a double $L_\infty$-algebra (resp. cyclic double $L_\infty$-algebra) structure.
\end{cor}
	
%
	
	\begin{remark}\ 
		
		\begin{itemize}
			\item [(i)]
		In fact, the assumption that $B$ is locally finite-dimensional in Theorem~\ref{Thm:From RB infinity alegbra to double Poisson infinity} is not essential. When $B$ is not locally finite-dimensional, the result still holds, and the proof can be carried out through direct computation.
			\item [(ii)] The construction in (\ref{Eq:Homotopy double Poisson construction}) serves as a  homotopy generalization of the construction in (\ref{double bracket}).
		\end{itemize}
	
	\end{remark}

\medskip

\subsection{Homotopy Rota-Baxter algebras and associative Yang-Baxter-infinity equation}\ 
	
	In \cite{Sched09}, Schedler introduced the notion of the \emph{associative Yang-Baxter-infinity equation} and established a one-to-one correspondence between homotopy double Lie algebra structures and skew-symmetric solutions of this equation.

		\begin{defn}\label{Def:HAYBE}\cite{Sched09}
				Let $A$ be a graded associative algebra. A solution of  {\bf associative Yang-Baxter-infinity equation}  is a family of elements $\left\{r_n \in A^{\otimes n}\right\}_{n \geqslant1}$ where each $r_n$ has degree  $ n-2$,  satisfying, for all $n \geqslant1$,
				$$
				\sum\limits_{i+j=n+1}(-1)^{(j+1)i} \sum\limits_{\sigma \in \mathfrak{C}_{n}}\sgn (\sigma) r_i^{\sigma(1), \sigma(2), \cdots, \sigma(i)} r_j^{\sigma(i), \sigma(i+1), \sigma(i+2), \cdots, \sigma(n)}=0.
				$$ 
				If, for all $n\geqslant1$, the element $r_n$ satisfies $\sgn(\sigma)r_n =r_n^{\sigma(1), \sigma(2), \cdots, \sigma(n)}$, then the solution is called   {\bf skew-symmetric}. 
			\end{defn}

		\begin{exam}
				Let $\{r_n\}_{n\geqslant1}$ be a skew-symmetric solution of associative Yang-Baxter-infinity equation. For small $n$, the associative Yang-Baxter-infinity equation yields the following:
				\begin{itemize}
						\item [(i)] When $n=1$, $|r_1|=-1$, $r_1\cdot r_1=0$, which implies that the operator $ \partial=[r_1,-]:A\rightarrow A$ defines a differential on $ A$;
						\item [(ii)] when $n=2$, $|r_1|=-1$, $|r_2|=0$,
						\[r_1^1\cdot r^{12}_2+r^{21}_2\cdot r^1_1= r_1^2\cdot r^{21}_2+r^{12}_2\cdot r_1^2,  \]
						which shows  that $r_2\in A\otimes A$ is a cycle with respect to the differential $[r_1,-]$;
						\item [(iii)] when $n=3$, 	$|r_1|=-1$, $|r_2|=0$,	$|r_3|=1$,
						\[ r^{12}_2\cdot  r^{23}_2+r^{23}_2\cdot  r^{31}_2+r^{31}_2\cdot  r^{12}_2 =  r_1^{1}\cdot   r_3^{123}+ r_1^{2} \cdot  r_3^{231}+ r_1^{3}\cdot   r_3^{312}+  r_3^{2 31}\cdot   r_1^{1}+  r_3^{ 312} \cdot  r_1^{2}+ r_3^{123}\cdot   r_1^{3},\]
						which shows that $r_2$ satisfies the usual associative Yang-Baxter equation up to homotopy provided by $r_3$.
					\end{itemize}
			\end{exam}
	
	Schedler further proved that there is a one-to-one correspondence between homotopy double Lie algebra structures  and skew-symmetric solutions to associative Yang-Baxter-infinity equation.
		\begin{prop}\label{Prop:Associative Yang-Baxter-infinity}\cite{Sched09}
				Let $V$ be a graded space. There is a bijection between  the set of homotopy double Lie algebra structures on $V$   and    skew-symmetric solutions   of the associative Yang-Baxter-infinity   equation    on $\End(V)$.
				\end{prop}

Combining Corollary~\ref{Coro: Pre-Calabi-Yau structure  on graded space induced by homotopy RB on endomorphism algebra}, Corollary~\ref{Coro: Correspongdence between some particular pre-Calabi-Yau and homotopy double Poisson} and Proposition~\ref{Prop:Associative Yang-Baxter-infinity}, we have the following equivalence:

\begin{prop}Let $V$ be  a finite dimensional graded space. Then the following data are equivalent:
	\begin{itemize}\label{RBS=Double Lie system}
		\item [\rm (i)] A fine special pre-Calabi-Yau algebra structure on $V$;
		\item [\rm (ii)] A  homotopy double Lie algebra  structure  $ \{\lbb -,\cdots,-\rbb\}_{n\geqslant 1}$ on $V$;
		
		\item [\rm (iii)] A differential $d$ on $V$ and an ultracyclic homotopy relative Rota-Baxter operator on dg algebra $(\mathrm{End}(V),[d,-])$;
		
		\item[\rm (iv)] A skew-symmetric solution to associative Yang-Baxter-infinity  equation in the graded algebra $\mathrm{End}(V)$.
		
	\end{itemize}
	\end{prop}
		
\medskip

\subsection{Homotopy Poisson structure arsing from homotopy Rota-Baxter algebras}\

It is well-know that the symmetric algebra of an $L_\infty$ carries a homotopy Poisson algebra.  Now we will show that this is also true for homotopy double Poisson algebras.
Let's recall some basics on homotopy Poisson algebras and homotopy double Poisson algebras following  \cite{CF07}.

\begin{defn}\ \label{Definiton: L-infinity algebra and homotopy Poisson algebra}
	\begin{itemize}
		\item[(i)] An {\bf $L_\infty$-algebra} is a graded vector space $L$ equipped with a collection of graded maps $\{l_n: L^{\otimes n} \to L\}_{n \geqslant1}$ of degree $|l_n| = n - 2$, satisfying the following conditions:
		\begin{itemize}
			\item[(1)] {\bf Skew-symmetry:} For all $\sigma \in \mathfrak{S}_n$,  
			\[
			l_n \circ \sigma^{-1} = \sgn(\sigma) \, l_n;
			\]
			
			\item[(2)] {\bf Generalized Jacobi identity:}
			\[
			\sum_{i=1}^n \sum_{\sigma \in \Sh(i, n-i)} \sgn(\sigma) (-1)^{i(n-i)} \,
			l_{n-i+1} \circ (l_i \otimes \id^{\otimes n-i}) \circ \sigma^{-1} = 0.
			\]
		\end{itemize}
		
		\smallskip
		
		\item[(ii)] A {\bf homotopy Poisson algebra} (also called a {\bf derived Poisson algebra}) is a graded vector space $L$ equipped with both an $L_\infty$-algebra structure $\{l_n\}_{n \geqslant1}$ and a graded commutative associative algebra structure, such that the following $\text{Leibniz}_\infty$ rule holds: for all $n \geqslant1$ and $x_1, \dots, x_{n-1}, x_n', x_n'' \in L$,
		\[
		l_n(x_1 \otimes \cdots \otimes x_n' x_n'') =
		l_n(x_1 \otimes \cdots \otimes x_n') \cdot x_n'' +
		(-1)^{|x_n'|(\sum_{i=1}^{n-1} |x_i| + n - 2)} x_n' \cdot l_n(x_1 \otimes \cdots \otimes x_{n-1} \otimes x_n'').
		\]
	\end{itemize}
\end{defn}

\begin{prop}\label{Thm:Derived Poisson manifold}
	Let $(V, \{\lbb-,\cdots,-\rbb_{n}\}_{n\geqslant1})$ be a homotopy double Lie algebra.
	Define a family of operations $\left\lbrace   l _n\right\rbrace _{n\geqslant1}$ on the graded symmetric algebra $S(V)$ as follows:  for all homogeneous elements $u_1^1,\cdots, u_{k_1}^1,\cdots,u_1^n,\cdots,u_{k_n}^n\in V$
	\begin{align*}
		l_n(u_1^1\cdots u_{k_1}^1\otimes\cdots\otimes u_1^n\cdots u_{k_n}^n) :=&(n-1)!\sum\limits_{1\leqslant q_1\leqslant k_1,\cdots 1\leqslant q_n\leqslant k_n}(-1)^{\sum\limits_{s=1}^n\left( \sum\limits_{t=1}^{s-1}(\sum\limits_{j=1}^{q_t-1}|u_j|+\sum\limits_{j=q_t+1}^{k_t}|u_j|)+\sum\limits_{j=1}^{q_s-1}|u_j|\right) |u_{q_s}|+\frac{(n-1)n}{2}}\\
		&\lbb  u_{q_1}^1,\cdots, u_{q_n}^n\rbb_n^{[1]}\cdots\lbb  u_{q_1}^1,\cdots,u_{q_n}^n\rbb^{[n]}_n\cdot u_1^1\cdots \widehat{u_{q_1}^1}\cdots u_{k_1}^1\cdots u_1^n \cdots \widehat{u_{q_n}^n}\cdots u_{k_n}^n.
	\end{align*}
	
	Then $\left( S(V), \left\lbrace   l _n\right\rbrace _{n\geqslant 1}\right) $ defines a  homotopy Poisson algebra. Thus  $V^\vee$ can be regarded as a formal derived Poisson manifold.
\end{prop}

\begin{proof}
	By the skew-symmetry and the $\text{Leibniz}_\infty$ rule satisfied by the homotopy double bracket, it follows that the operators $\{l_n\}_{n \geqslant 1}$ are well-defined on the symmetric algebra $S(V)$. Moreover, it is straightforward to verify that the brackets $\{l_n\}_{n \geqslant 1}$ inherit skew-symmetry and satisfy the $\text{Leibniz}_\infty$ rule with respect to the natural multiplication on $S(V)$. Therefore, it remains only to check that they also satisfy the $\text{Jacobi}_\infty$ rule.
	
	Since each operation $l_n$ satisfies the $\text{Leibniz}_\infty$ rule, it suffices to verify that the family $\{l_n\}_{n \geqslant 1}$ satisfies the $\text{Jacobi}_\infty$ identity on the generating space $V \subset S(V)$. Let $x_1, \dots, x_n \in V$, and let $\mu$ denote the natural multiplication on the symmetric algebra $S(V)$. Then, using the skew-symmetry of the brackets $\{\lbb -, \dots, - \rbb_n\}_{n \geqslant 1}$ and applying Lemma~\ref{Lemma:def homotopy double Lie2}, we obtain:
	{\small	\begin{align}
			& \sum\limits\limits_{i=1}^n\sum\limits\limits_{\sigma\in \Sh(i,n-i)}\chi(\sigma; x_1, \dots, x_n)(-1)^{i(n-i)}l_{n-i+1}(l_i(x_{\sigma(1)}\ot \cdots \ot x_{\sigma(i)})\ot x_{\sigma(i+1)}\ot \cdots \ot x_{\sigma(n)}) \label{Eq: Derived Poisson manifold}\\
			=& \sum\limits\limits_{i=1}^n(i-1)!(n-i)!\sum\limits\limits_{\sigma\in \Sh(i,n-i)}\sum\limits\limits_{k=1}^{i-1} \sgn(\sigma) (-1)^{i(n-i)+\frac{i(i-1)+(n-i)(n-i+1)}{2}}\mu\circ \Big(\id^{k-1}\otimes\lbb-,\cdots,-\rbb^{[1]}_{n-i+1}\otimes\id^{i-k}\nonumber\\
			&\otimes \lbb-,\cdots,-\rbb^{[2]}_{n-i+1} \otimes\cdots\otimes \lbb-,\cdots,-\rbb^{[n-i+1]}_{n-i+1}\Big)\circ \left(\lbb-,\cdots,-\rbb_{i}\otimes \id^{\ot n-i} \right)\circ\sigma^{-1} (x_1\otimes\cdots\otimes x_n)\nonumber\\
			=&\sum\limits_{i=1}^n (-1)^{i(n-i)+\frac{i(i-1)+(n-i)(n-i+1)}{2}}(i-1)!(n-i)! \sum\limits_{\sigma\in \Sh(i,n-i)}\sum\limits_{\tau\in \mathfrak{C}_i\times \id^{ n-i}} \sgn (\sigma)\sgn (\tau) \nonumber\\
			&\mu\circ(\id^{\otimes i-1}\otimes \lbb-,\cdots,-\rbb_{n-i+1})\circ (\lbb-,\cdots,-\rbb_{i}\otimes \id^{n-1}) \circ \tau^{-1}  \circ\sigma^{-1}(x_1\otimes\cdots\otimes x_n).\nonumber
	\end{align}}
	
	Note that, for each $1\leqslant i\leqslant n$, the composite map \[\mu\circ(\id^{\otimes i-1}\otimes \lbb-,\cdots,-\rbb_{n-i+1})\circ (\lbb-,\cdots,-\rbb_{i}\otimes \id^{n-1}) \]  is graded symmetric with respect to the first $i-1$ inputs  and the last $n-i$ inputs. Thus, 
	{\small	\begin{align*}
			(\ref{Eq: Derived Poisson manifold})	=&\sum\limits_{i=1}^n (-1)^{i(n-i)+\frac{i(i-1)+(n-i)(n-i+1)}{2}} \sum\limits_{\sigma\in \mathfrak{S}_n} \sgn (\sigma)\mu\circ(\id^{\otimes i-1}\otimes \lbb-,\cdots,-\rbb_{n-i+1})\circ (\lbb-,\cdots,-\rbb_{i}\otimes \id^{n-1}) \circ \sigma^{-1}\\
			=&\sum\limits_{i=1}^n (-1)^{i(n-i)} \sum\limits_{\sigma\in \mathfrak{S}_n} \sgn (\sigma)\mu\circ(\id^{\otimes i-1}\otimes \lbb-,\cdots,-\rbb_{n-i+1}^{\rm{op}})\circ (\lbb-,\cdots,-\rbb_{i}^{\rm{op}}\otimes \id^{n-1})\circ \sigma^{-1}\\
			=&\sum\limits_{i=1}^n (-1)^{i(n-i)}\sum\limits_{\sigma\in \mathfrak{S}_n} \sgn (\sigma) \mu\circ\left( \lbb\lbb-,\cdots,-\rbb^{\rm{op}}_{i},-,\cdots,-\rbb^{\rm{op}}_{R,n-i+1}\right) \sigma^{-1}(x_1\otimes\cdots\otimes x_n)\\
			=&\sum\limits_{\tau\in \mathfrak{S}_{n-1}\times \id}\mu\circ\left( \sum\limits_{i=1}^n (-1)^{i(n-i)}\sum\limits_{\sigma\in \mathfrak{C}_n} \sgn (\sigma)\sigma\cdot \left( \lbb\lbb-,\cdots,-\rbb^{\rm{op}}_{i},-,\cdots,-\rbb^{\rm{op}}_{R,n-i+1}\right) \sigma^{-1}\right)\cdot\tau^{-1}(x_1\otimes\cdots\otimes x_n) \\
			=&0.
	\end{align*}}

	This completes the proof that the operations $\{l_n\}_{n \geqslant1}$ satisfy the $\text{Jacobi}_\infty$ identity.
\end{proof}

\begin{prop}
	Let $A$ be a finite-dimensional dg algebra, and let $B$ be a locally finite-dimensional dg left $A$-module. Suppose there exists an ultracyclic homotopy relative Rota-Baxter operator $\{T_n: (A^\vee)^{\otimes n} \rightarrow A\}_{n \geqslant 1}$. Then the symmetric algebra $S(B)$ inherits a homotopy Poisson algebra structure. In particular, the graded dual $B^\vee$ can be regarded as a formal derived Poisson manifold.
\end{prop}

\bigskip

\bigskip

\textbf{Acknowledgements} 
The   authors  were supported by  the National Key R$\&$D Program of China (No. 2024YFA1013803), by Key Laboratory of  Mathematics and Engineering Applications (Ministry of Education), and by Shanghai Key Laboratory of PMMP (No. 22DZ2229014).
  The first author was also financed by the project OZR3762 of Vrije Universiteit Brussel, by the FWO Senior Research Project G004124N. 
 

\textbf{Conflict of Interest}

None of the authors has any conflict of interest in the conceptualization or publication of this
work.

\textbf{Data availability}

Data sharing is not applicable to this article as no new data were created or analyzed in this study.

\bigskip

	\appendix
	
	\section{Proof of Proposition~\ref{Proposition: dual homotopy Rota-Baxter module}}\label{Appendix:  Proof of dual homotopy Rota-Baxter module}
	\begin{proof}
	We just to check that $\{T_{i,j}^{M^\vee}\}_{i,j\geqslant  0}$ satisfies Equation~\eqref{Defn: homotopy Rota-Baxter module}, that is, we need to
	check the following identity: 
	{\small\begin{eqnarray*}
			&&\underbrace{\sum\limits\limits_{{ i_1+\dots+i_p+l=m,\atop j_1+\cdots+j_q+k=n}\atop
					p,q,l,k\geqslant  0 } (-1)^{\alpha}m_{p,q}^{M^\vee}\circ\Big(T_{i_1}\ot \cdots \ot T_{i_p}\ot T^{M^\vee}_{l,k} \ot T_{j_1}\otimes \cdots \ot T_{j_q}\Big)}_{\mathrm{(I)}}\\
			&=&\underbrace{\sum\limits\limits_{{{{i_1+\cdots+i_p+l=m,\atop j_1+\cdots+j_q+k=n,}}\atop i_1,\cdots,i_p,j_1,\cdots,j_q\geqslant  1}\atop p,\ q,\ l,\ k\geqslant 0;} (-1)^{\beta_1}
				T^{M^\vee}_{l,k}\circ\Big(\id_A^{\ot l}\ot m_{p, q}^{M^\vee}\circ( T_{i_1}\ot   \cdots\ot T_{i_p}\ot\id_M \ot T_{j_{1}}\ot \cdots\ot T_{j_q})\ot \id_A^{\ot k}\Big)}_{\mathrm{(II)}}\\
			&&\nonumber+\underbrace{\sum\limits_{{{i_1+\cdots+i_p+l+1=m\atop j_1+\cdots+j_q+k=n} \atop i_1,\cdots,i_p,j_1,\cdots,j_q\geqslant 0 }  \atop v,l,k,p,q\geqslant  0}  (-1)^{\beta_2}
				T_{l,k}^{M^\vee}\circ\Big(\id_A^{\ot l}\ot m^{M^\vee}_{p+1,q}\circ( T_{i_1}\ot \cdots \ot T_{i_v}\ot \id_A\ot T_{i_{v+1}}  \cdots\ot T_{i_p}\ot T^{M^\vee}_{r,t} \ot T_{j_{1}}\ot \cdots\ot T_{j_q})\ot \id_A^{\ot k}\Big)}_{\mathrm{(III)}}\\
			&&\nonumber+\underbrace{\sum\limits_{{{i_1+\cdots+i_p+l+1=m\atop j_1+\cdots+j_q+k=n} \atop i_1,\cdots,i_p,j_1,\cdots,j_q\geqslant 0 }  \atop v,l,k,p,q\geqslant  0}  (-1)^{\beta_3}
				T_{l,k}^{M^\vee}\circ\Big(\id_A^{\ot i}\ot m^{M^\vee}_{p, q+1}\circ( T_{i_1}\ot   \cdots\ot T_{i_p}\ot T^{M^\vee}_{r,t} \ot T_{j_{1}}\ot \cdots  \ot T_{j_v}\ot \id_A\ot T_{j_{v+1}}\ot\cdots  T_{j_q})\ot \id_A^{\ot k}\Big)}_{\mathrm{(IV)}}.
	\end{eqnarray*}}

	\textbf{Term (I):} 
	{\small	\begin  {align*}
		&-\sum\limits\limits_{{ i_1+\dots+i_p+l=m,\atop j_1+\cdots+j_q+k=n}\atop
			p,q,l,k\geqslant  0 } (-1)^{\alpha}m_{p,q}^{M^\vee}\circ\Big(T_{i_1}\ot \cdots \ot T_{i_p}\ot T^{M^\vee}_{l,k} \ot T_{j_1}\otimes \cdots \ot T_{j_q}\Big)(a_1\otimes\cdots\otimes a_m\otimes f\otimes b_1\otimes\cdots\otimes b_n )\\
		= &\sum\limits\limits_{{{{i_1+\cdots+i_p+l=m,\atop j_1+\cdots+j_q+k=n,}}\atop i_1,\cdots,i_p,j_1,\cdots,j_q\geqslant  1}\atop p,\ q,\ l,\ k\geqslant 0;} (-1)^{\beta_1+\theta}
		f\circ	T^M_{l,k}\circ\Big(\id_A^{\ot l}\ot m_{p, q}\circ( T_{i_1}\ot   \cdots\ot T_{i_p}\ot\id_M \ot T_{j_{1}}\ot \cdots\ot T_{j_q})\ot \id_A^{\ot k}\Big)\\
		&\qquad\qquad\qquad(b_1\otimes\cdots\otimes b_n\otimes \id_M\otimes  a_1\otimes\cdots\otimes a_m), 
\end{align*}	}

\textbf{Term (II):}

{\small	\begin{align*}
		&\sum\limits\limits_{{{{i_1+\cdots+i_p+l=m,\atop j_1+\cdots+j_q+k=n,}}\atop i_1,\cdots,i_p,j_1,\cdots,j_q\geqslant  1}\atop p,\ q,\ l,\ k\geqslant 0;} (-1)^{\beta_1}
		T^{M^\vee}_{l,k}\circ\Big(\id_A^{\ot l}\ot m_{p, q}^{M^\vee}\circ( T_{i_1}\ot   \cdots\ot T_{i_p}\ot\id_M \ot T_{j_{1}}\ot \cdots\ot T_{j_q})\ot \id_A^{\ot k}\Big)\\
		&\qquad\quad\qquad\qquad\qquad(a_1\otimes\cdots\otimes a_m\otimes f\otimes b_1\otimes\cdots\otimes b_n )	\\
		=	&-\sum\limits\limits_{{ i_1+\dots+i_p+l=m,\atop j_1+\cdots+j_q+k=n}\atop
			p,q,l,k\geqslant  0 } (-1)^{\alpha+\theta}f\circ m_{p,q}\circ\Big(T_{i_1}\ot \cdots \ot T_{i_p}\ot T^M_{l,k} \ot T_{j_1}\otimes \cdots \ot T_{j_q}\Big)(b_1\otimes\cdots\otimes b_n\otimes \id_M\otimes  a_1\otimes\cdots\\
		&\qquad\qquad\qquad\qquad\cdots\otimes a_m)(b_1\otimes\cdots\otimes b_n\otimes \id_M\otimes  a_1\otimes\cdots\otimes a_m),
\end{align*}}

\textbf{Term (III):}
{\small \begin{align*}
		&\sum\limits_{{{i_1+\cdots+i_p+l+1=m\atop j_1+\cdots+j_q+k=n} \atop i_1,\cdots,i_p,j_1,\cdots,j_q\geqslant 0 }  \atop v,l,k,p,q\geqslant  0}  (-1)^{\beta_2}
		T_{l,k}^{M^\vee}\circ\Big(\id_A^{\ot l}\ot m^{M^\vee}_{p+1,q}\circ\Big( T_{i_1}\ot \cdots \ot T_{i_v}\ot \id_A\ot T_{i_{v+1}}  \cdots\ot T_{i_p}\ot T^{M^\vee}_{r,t} \ot T_{j_{1}}\ot\cdots\\
		& \qquad\qquad\qquad\qquad\qquad \cdots\ot T_{j_q}\Big)\ot \id_A^{\ot k}\Big)(a_1\otimes\cdots\otimes a_m\otimes f\otimes b_1\otimes\cdots\otimes b_n )	\\
		=&\sum\limits_{{{i_1+\cdots+i_p+l+1=m\atop j_1+\cdots+j_q+k=n} \atop i_1,\cdots,i_p,j_1,\cdots,j_q\geqslant 0 }  \atop v,l,k,p,q\geqslant  0}  (-1)^{\beta_3+\theta}
		f\circ	T_{l,k}^M\circ\Big(\id_A^{\ot i}\ot m_{p, q+1}\circ\Big( T_{i_1}\ot   \cdots\ot T_{i_p}\ot T^M_{r,t} \ot T_{j_{1}}\ot \cdots \ot T_{j_v}\ot \id_A\ot T_{j_{v+1}}\ot\cdots \\
		&\nonumber \quad\quad\quad\quad\quad\quad\quad\quad\quad\quad\cdots  T_{j_q}\Big)\ot \id_A^{\ot k}\Big)(b_1\otimes\cdots\otimes b_n\otimes \id_M\otimes  a_1\otimes\cdots\otimes a_m),
\end{align*}}

\textbf{Term (IV):}
{\small \begin{align*}
		&\sum\limits_{{{i_1+\cdots+i_p+l+1=m\atop j_1+\cdots+j_q+k=n} \atop i_1,\cdots,i_p,j_1,\cdots,j_q\geqslant 0 }  \atop v,l,k,p,q\geqslant  0}  (-1)^{\beta_3}
		T_{l,k}^{M^\vee}\circ\Big(\id_A^{\ot i}\ot m^{M^\vee}_{p, q+1}\circ( T_{i_1}\ot   \cdots\ot T_{i_p}\ot T^{M^\vee}_{r,t} \ot T_{j_{1}}\ot\cdots \ot T_{j_v}\ot \id_A\ot T_{j_{v+1}}\ot\cdots \\
		&\quad\quad\quad\quad\quad\quad\quad\quad\quad\quad \cdots  T_{j_q})\ot \id_A^{\ot k}\Big) (a_1\otimes\cdots\otimes a_m\otimes f\otimes b_1\otimes\cdots\otimes b_n )	\\
		=&\sum\limits_{{{i_1+\cdots+i_p+l+1=m\atop j_1+\cdots+j_q+k=n} \atop i_1,\cdots,i_p,j_1,\cdots,j_q\geqslant 0 }  \atop v,l,k,p,q\geqslant  0}  (-1)^{\beta_2+\theta}
		f\circ T_{l,k}^M\circ\Big(\id_A^{\ot l}\ot m_{p+1,q}\circ \Big( T_{i_1}\ot \cdots \ot T_{i_v}\ot \id_A\ot T_{i_{v+1}} \otimes   \cdots\ot T_{i_p}\ot T^M_{r,t} \ot T_{j_{1}}\ot\cdots\\
		&  \quad\quad\quad\quad\quad\quad\quad\quad\quad\quad  \cdots\ot T_{j_q}\Big)\ot \id_A^{\ot k}\Big)(b_1\otimes\cdots\otimes b_n\otimes \id_M\otimes  a_1\otimes\cdots\otimes a_m),
\end{align*}}
where 
\[\theta=(\sum_{s=1}^m|a_s|)(\sum_{s=1}^n|b_s|)+|f|(\sum_{s=1}^m|a_s|+m+n+1)+(m+n+1)(n+1).\]

Taking the sum, one can easily see that 
\[{\bf{(I)+(II)+(III)+(IV)}}=0,\] 
since $\{T_{i,j}^M\}_{i,j\geqslant 0}$ subjects to Equation~\eqref{Defn: homotopy Rota-Baxter module}.
Thus  $\{T_{i,j}^{M^\vee}\}_{i,j\geqslant 0}$ satisfies Equation~\eqref{Defn: homotopy Rota-Baxter module}.
\end{proof}

\bigskip

	\section{Proof of   Lemma~\ref{Lem: A_infinity structures on homotopy RB dg algebras}}\label{Appendix: Proof of Lemma: A_infinity structures on homotopy RB dg algebras}
	
Before proving Lemma~\ref{Lem: A_infinity structures on homotopy RB dg algebras}, we first introduce the following  lemma, which will be used extensively in the proof.
		\begin{lem}\label{Lemma: kappa morphism}
		Let $(A,d_A,\cdot)$ be a dg algebra and $ (B,{\rhd})$   a left  dg $A$-module.	  Then  
		\[ \kappa:B\otimes B^\vee\rightarrow A^\vee\]
   is a dg $A$-bimodule morphism. 
   
   Moreover, if  $(A, B)$ is an interactive pair, then $\kappa$ is also a right dg $B$-module morphism.
	\end{lem}
	
	\begin{proof}
		For any $b\in B$, $f\in B^\vee$ and $a_1,a_2\in A$,
		\begin{eqnarray*}\kappa\left( (a_2{\rhd}b)\otimes f\right) (a_1)&=&(-1)^{(|a_2|+|b|(|f|+|a_1|)}f\left( (a_1\cdot a_2){\rhd}b\right)\\&=&(-1)^{|a_2|(|b|+|f|+|a_1|)}\kappa(b\otimes f)(a_1\cdot a_2)\\&=&\left(a_2\rhd\kappa(b\ot f) \right)(a_1).  
		\end{eqnarray*}
		Similarly, we also have \[\kappa(b\otimes (f{\lhd}a_2))= \kappa(b\otimes f)\lhd a_2, d_{A^\vee}\left(\kappa(b\otimes f)\right)= \kappa\left( d_B(b)\otimes f \right)+(-1)^{|b|}\kappa\left(b\otimes d_{B^\vee}(f) \right).\]
		Thus, $\kappa$ is a dg $A$-bimodule morphism.
	
		Now, we assume that $(A, B)$ is an interactive pair. For any $b_1,b_2\in B$, $f\in B^\vee$ and $a\in A$,
			\begin{eqnarray*}\kappa\left(b_1\otimes f\btl b_2\right) (a)&=&(-1)^{|b_1|(|f|+|b_2|+|a|}(f \btl b_2)\left( a{\rhd}b_1\right)\\&=&(-1)^{|b_1|(|f|+|b_2|+|a|}f\left(b_2*( a{\rhd}b_1)\right)\\
				&=&(-1)^{|b_1|(|f|+|b_2|+|a|}f\left((b_2\btr a{\rhd})b_1\right)\\&=&\left(\kappa(b_1\otimes f)\btl b_2\right) (a),  
		\end{eqnarray*}
		where ``$*$" stands for the multiplication on $B$ and ``$\btl$" stands for the induced right action of $B$ on $A^\vee$.
	Thus,  $\kappa$ is also a right dg $B$-module morphism.
	\end{proof}

	\begin{proof}[Proof of Lemma~\ref{Lem: A_infinity structures on homotopy RB dg algebras}]  We proceed to verify that the Stasheff identities for the operations $\{m_n\}_{n\geqslant  1} $, introduced in Lemma~\ref{Lem: A_infinity structures on homotopy RB dg algebras}, hold trivially in every case. We divide it into the following five cases.

			\textbf{Case I}: for $b_1,\dots,b_{n+1}\in B$ and $f_1,\dots,f_n\in B^{\vee} $, by Lemma~\ref{Lemma: kappa morphism} , we have
{\small{
		\begin{align*}
			&\sum_{    i+j +k= 2n+1,\atop i,k\geqslant 0,j\geqslant 1}(-1)^{i+jk}m_{i+k+1}(\Id^{\otimes i}\otimes m_j\otimes \Id^{\otimes k})(b_1\otimes s^{-1}f_1\otimes \cdots\otimes s^{-1}f_n\otimes b_{n+1})\\
		=&\sum_{    i+j = n,\atop {s+k=2i,
				\atop i, j ,k\geqslant  0}  }  m_{2i+1}\circ\Big(\id^{\ot s}\ot m_{2j+1}\ot \id^{\ot k}\Big)(b_1\otimes s^{-1}f_1\otimes \cdots\otimes s^{-1}f_n\otimes b_{n+1})\\
			 =&\sum_{    i+j = n,\atop  
				{i, j  \geqslant  1,\atop i-1\geqslant p\geqslant 0} } (-1)^{p+\sum_{k=1}^p(|b_k|+|f_k|)}m_{2i+1}\Big(b_1\otimes s^{-1}f_1\otimes \cdots\otimes b_p\otimes s^{-1}f_p\otimes m_{2j+1}\big(b_{p+1}\otimes s^{-1}f_{s+j}\otimes b_{s+j+1}\big)\\
				&\quad \quad\quad \quad\otimes s^{-1}f_{p+j+1}\otimes \cdots\otimes b_{n+1} \Big) \\
			&+\sum_{     i+j = n,\atop  
				{i, j  \geqslant  1,\atop i-1\geqslant p\geqslant 0} } (-1)^{p+\sum\limits_{k=1}^p(|b_k|+|f_k|)+|b_{p+1}|}m_{2i+1}\Big(b_1\otimes s^{-1}f_1\otimes \cdots\otimes s^{-1}f_{p }\otimes b_{p+1}\otimes m_{2j+1}\big(  s^{-1}f_{p+1}\otimes \cdots\otimes s^{-1}f_{p+j+1}\big) \\
				&\ \quad\qquad \quad \otimes b_{p+j+2}\otimes \cdots\otimes b_{n+1}\Big)\\
			&+\sum_{    i+j = n,\atop  
				i, j  \geqslant  1  }(-1)^{i+\sum\limits_{k=1}^i(|b_k|+|f_k|)}m_{2i+1}\Big(b_1\otimes s^{-1}f_1\otimes \cdots\otimes b_i\otimes s^{-1}f_i\otimes m_{2j+1}(b_{i+1}\otimes s^{-1}f_{i+1}\otimes\cdots\otimes s^{-1}f_{n-1}\otimes b_{n+1})  \Big) \\
			&+\sum_{      
				0\leqslant p\leqslant n-1  } (-1)^{p+\sum\limits_{k=1}^p(|b_k|+|f_k|)+1}m_{2n+1}\Big(b_1\otimes s^{-1}f_1\otimes \cdots\otimes b_p\otimes  s^{-1}f_p\otimes d_B(b_{p+1})\otimes s^{-1}f_{p +1}\otimes \cdots\otimes b_{n+1}\Big) \\
			&+\sum_{     
				0\leqslant p\leqslant n-1} (-1)^{p+\sum\limits_{k=1}^p(|b_k|+|f_k|)+|b_{p+1}|}m_{2n+1}\Big(b_1\otimes s^{-1}f_1\otimes \cdots\otimes s^{-1}f_{p }\otimes b_{p+1}\otimes  s^{-1} d_{B^{\vee} }(  f_{p+1} )\otimes \cdots\otimes b_{n+1}\Big)\\
			&-d_B m_{2n+1}\Big(b_1\otimes s^{-1}f_1\otimes \cdots\otimes s^{-1}f_n\otimes b_{n+1}\Big)+(-1)^{n-1+\sum\limits_{k=1}^n(|b_k|+|f_k|)} m_{2n+1}\Big(b_1\otimes s^{-1}f_1\otimes\cdots \otimes s^{-1}f_n\otimes d_B(b_{n+1})\Big)\\
			=&\sum_{    i+j = n;\atop  
				{ i, j \geqslant 1;\atop  i-1\geqslant p\geqslant 0  }} (-1)^{\gamma_1}m_{2i+1}\Big(b_1\otimes s^{-1}f_1\otimes  \cdots\otimes b_p\otimes s^{-1}f_p\otimes T_{j}\Big(\kappa(b_{p+1}\otimes f_{p+1})\otimes\cdots\otimes\kappa(b_{p+j}\otimes f_{p+j})\Big) {\rhd}b_{p+j+1}\\
				& \qquad \qquad \otimes s^{-1}f_{p+j+1} \otimes \cdots\otimes b_{n+1}\Big)\\
			&+\sum_{      i+j = n;\atop  
					{ i, j \geqslant 1;\atop  i-1\geqslant p\geqslant 0  }} (-1)^{\gamma_2} m_{2i+1}\Big(b_1\otimes s^{-1}f_1\otimes \cdots\otimes  s^{-1}f_p\otimes b_{p+1}\otimes s^{-1}f_{p+1} {\lhd}T_{j}\Big(\kappa(b_{p+2}\otimes f_{p+2})\otimes  \cdots\otimes \kappa(b_{p+j+1}\otimes f_{p+j+1})\Big) \\
				&\qquad \qquad  \otimes s^{-1}f_{p+j+2} \otimes \cdots\otimes b_{n+1}\Big)\\
		&+	\sum_{    i+j = n,\atop  
				i, j  \geqslant  1  }(-1)^{\gamma_3} m_{2i+1}\left(b_1\otimes s^{-1}f_1\otimes\cdots\otimes b_i\otimes s^{-1}f_i\otimes T_{j}(\kappa(b_{i+1}\otimes f_{i+1})\otimes \cdots\otimes \kappa(b_{n}\otimes f_{n})){\rhd}b_{n+1}\right)\\
	&	+	\sum_{      
				0\leqslant p\leqslant n-1  } (-1)^{n-1+\sum\limits_{k=1}^p(|b_k|+|f_k|)  +\gamma}  T_n\left(\kappa(b_1\otimes f_1)\otimes\cdots\otimes\kappa(d_B(b_{p+1})\otimes f_{p+1})\otimes\cdots \otimes\kappa(b_n,f_n) \right) {\rhd}b_{n+1}\\
	&	+	\sum_{      
				0\leqslant p\leqslant n-1  } (-1)^{n-1+\sum\limits_{k=1}^p(|b_k|+|f_k|)  +\gamma+|b_{p+1}|} T_n\Big(\kappa(b_1\otimes f_1)\otimes \cdots\otimes \kappa(b_{p+1}\otimes d_{B^{\vee} }(f_{p+1}))\otimes \cdots\otimes \kappa(b_n\otimes f_n) \Big) {\rhd}b_{n+1}\\
			&- 
			(-1)^{\gamma}d_B\Big( T_n\left(\kappa(b_1\otimes f_1)\otimes\cdots\otimes \kappa(b_n\otimes f_n) \right){\rhd}b_{n+1} \Big)	\\
	&+(-1)^{n-1+\sum\limits_{k=1}^n(|b_k|+|f_k|)+\gamma} T_n\Big(\kappa(b_1\otimes f_1)\otimes\cdots\otimes \kappa(b_n\otimes f_n) \Big){\rhd}d_B(b_{n+1}) \\
		=	&\sum_{     i+j = n;\atop  
				{ i, j \geqslant 1;\atop  i-1\geqslant p\geqslant 0  }}   (-1)^{\gamma_4} T_i\Big(\kappa(b_1\otimes f_1)\otimes\cdots\otimes\kappa(b_s\otimes  f_s)\otimes m^l\Big( T_{j}\Big(\kappa(b_{s+1}\otimes f_{s+1})\otimes\cdots \otimes \kappa (b_{s+j}\otimes f_{s+j})\Big)\otimes\kappa (b_{s+j+1}\otimes f_{s+j+1})\Big)\\
				&\qquad\qquad \otimes\cdots\otimes\kappa(b_n\otimes f_n)\Big){\rhd}b_{n+1} \\
	&	+	\sum_{    i+j = n,\atop  
				{i, j  \geqslant  1\atop {i-1 \geqslant s\geqslant 0  }}} (-1)^{\gamma_5}T_i\Big(\kappa(b_1\otimes f_1)\otimes\cdots\otimes\kappa(b_s\otimes  f_s)\otimes m^r \Big(\kappa (b_{s+1}\otimes f_{s+1})\otimes   T_{j}\Big(\kappa(b_{s+2}\otimes f_{s+2})\otimes\cdots\\
			& \quad\qquad\qquad\qquad\cdots\otimes \kappa(b_{s+j+1},f_{s+j+1})\Big) \Big)  \otimes\cdots \otimes\kappa(b_n\otimes f_n)\Big){\rhd}b_{n+1}\\
			&-\sum_{    i+j = n,\atop  
				i, j  \geqslant  1  }(-1)^{\gamma_6}m \Big(T_i\Big(\kappa(b_1\otimes f_1)\otimes\cdots\otimes\kappa(b_i\otimes  f_i)\Big)\otimes    T_{j}\Big(\kappa(b_{i+1}\otimes  f_{i+1})\otimes \cdots\otimes\kappa(b_{n}\otimes f_{n})\Big)\Big)   {\rhd}b_{n+1} \\
	&-	(-1)^{\gamma}d_A\Big(T_n\Big(\kappa(b_1\otimes f_1)\otimes\cdots\otimes \kappa(b_n\otimes f_n) \Big) \Big) {\rhd}b_{n+1} 	\\
		&+(-1)^{n-1+\sum\limits_{k=1}^s(|b_k|+|f_k|)+\gamma}T_n\Big(\kappa(b_1\otimes f_1)\otimes\cdots\otimes d_{A^{\vee} }(\kappa(b_s\otimes f_s))\otimes\cdots\otimes\kappa(b_n\otimes f_n) \Big){\rhd}b_{n+1} \\
		= 	& (-1)^{\gamma} \Bigg(\Big( \sum_{s+k+j+1=n} (-1)^{s+(j-1)(i-s)} T_{i}\circ\left( \id^{\ot s} \ot m^l\circ(T_{j}\otimes \id)\ot\id^{\ot k}\right)\\
		&\qquad+\sum_{s+k+j+1=n} (-1)^{s+(j-1)(i-s)}(-1)^{1-j} T_{i}\circ\big( \id^{\ot s} \ot m^r\circ( \id\otimes T_{j})\ot\id^{\ot k}\big) -\sum_{i+j=n}(-1)^{1+i} m\circ\big(T_{i}\ot  T_{j}\big) \\
			&\qquad + \sum_{s+k+1=n} (-1)^{n-1}T_n\circ\big( \id^{\ot s}\otimes  d_{A^{\vee} }\otimes\id^{\ot k}\big)-d_A\circ T_{n}\Big)\\
			&    \qquad \Big(\kappa(b_1\otimes f_1)\otimes\cdots\otimes \kappa(b_n\otimes f_n) \Big) \Bigg) {\rhd}b_{n+1}\\
			=&0,
		\end{align*}
	}}
where 
{\small
\begin{align*}
	\gamma =& \sum\limits_{k=1}^n (n - k + 1)|b_k| + \sum_{k=1}^n (n - k)|f_k|;\\
	 \gamma_1=&p+\sum\limits_{k=1}^p(|b_k|+|f_k|)+\sum\limits_{k=p+1}^{p+j}(p+j-k+1)|b_k|+\sum\limits_{k=p+1}^{p+j}(p+j-k)|f_k|;\\
	 \gamma_2=&p+\sum_{k=1}^{p+1}(|b_k|+|f_k|)+\sum_{k=p+2}^{p+j+1}(p+j-k+1)|b_k|+\sum_{k=p+1}^{p+j}(p+j-k)|f_k|;\\ 
	 \gamma_3=&i+\sum\limits_{k=1}^i(|b_k|+|f_k|)+\sum\limits_{k=i+1}^n(n-k+1)|b_k|+\sum\limits_{k=i+1}^n(n-k)|f_k|;\\
	 \gamma_4=&p+(j-1)(i-p)+(j-1)(\sum_{k=1}^p(|b_k|+|f_k|))+\sum_{k= 1}^{p+j}(n-k+1)|b_k|+\sum_{k= 1}^{p+j}(n-k)|f_k|;\\
	 \gamma_5=&s+(j-1)(i-s-1)+(j-1)(\sum_{k=1}^s(|b_k|+|f_k|))+\sum_{k= 1}^{s+j}(n-k+1)|b_k|+\sum_{k= 1}^{s+j}(n-k)|f_k|;\\
	 \gamma_6=&i+1+(j-1)(\sum_{k=1}^i(|b_k|+|f_k|))+\sum_{k=i+1}^n(n-k+1)|b_k|+\sum_{k=i+1}^n(n-k)|f_k|.
\end{align*}}

\smallskip

\textbf{Case II}: for brevity, we omit the detailed calculations, which are analogous to those in Case I.	For $b_1,\dots,b_{n+1}\in B$ and $f_0,\dots,f_n\in B^{\vee} $,
	{\small	\begin{eqnarray*}	 
			&&	\sum_{    i+j = n,\atop s+k=2i,
				i, j ,k\geqslant  0  }  m_{2i+1}\circ\Big(\id^{\ot s}\ot m_{2j+1}\ot \id^{\ot k}\Big)(s^{-1}f_0\otimes b_1\otimes s^{-1}f_1\otimes\cdots\otimes b_n\otimes s^{-1}f_n ) \\
			&&=  (-1)^{\gamma} s^{-1}f_0{\lhd}\Bigg( \Big(- \sum_{s+k+j+1=n} (-1)^{s+(j-1)(i-s)}  T_{i}\circ\left( \id^{\ot s} \ot m_l\circ(T_{j}\otimes \id)\ot\id^{\ot k}\right)    \\
			&& -  \sum_{s+k+j+1=n} (-1)^{s+(j-1)(i-s-1)} T_{i}\circ\left( \id^{\ot s} \ot m_r\circ( \id\otimes T_{j})\ot\id^{\ot k}\right)+\sum_{i+j=n}(-1)^{1+i} m\circ\Big(T_{i}\ot  T_{j}\Big)  \\
			&&   -\sum_{s+k+1=n} (-1)^{n-1}T_n\circ\left( \id^{\ot s}\otimes  d_{A^{\vee} }\otimes\id^{\ot k}\right)+d_A\circ T_{n}\Big) \Big( \kappa(b_1\otimes f_1)\otimes \cdots\otimes  \kappa(b_n\otimes f_n)\Big) \Bigg)  \\
			&&=0.
		\end{eqnarray*}}
	Since $(A,B)$ is an interactive pair, by Lemma~\ref{Lemma: kappa morphism},  then     
	\begin{align}\label{Eq: Left interactive pair on Ainfinty alg}
		 \kappa(b_1\otimes f){\btl} b_2=\kappa\big(b_1\otimes sm_2(s^{-1}f\otimes b_2)\big),~\forall b_1,b_2\in B, ~f\in B^\vee.
	\end{align}
 Next,	we will use Equation~(\ref{Eq: Left interactive pair on Ainfinty alg}) to  verify three cases where the Stasheff identity holds with a nontrivial $m_2$ involved.
	
	\smallskip
	
	\textbf{Case III}: for  $n\geqslant1$, $b_1,\dots,b_{n+2}\in B$  and $f_1,\dots,f_{n}\in B^\vee$, 
	\begin{align*}	
		&\sum\limits_{    i+j+k= 2n+2,\atop
			i, k\geqslant  0, j\geqslant  1 } (-1)^{i+jk}m_{i+1+k}\circ\Big(\id^{\ot i}\ot m_j\ot \id^{\ot k}\Big)\Big( b_1\otimes s^{-1}f_1\otimes\dots\otimes s^{-1} f_{n}\otimes b_{n+1}\otimes   b_{n+2}\Big)\\
		=&\Big( m_{2n+1}\circ (\id^{\otimes 2n}\otimes m_2)-m_{2n+1}\circ (\id^{\otimes 2n-1}\otimes m_2\otimes \id)-m_2(m_{2n+1}\otimes \id)\Big)\\
		&\qquad\Big( b_1\otimes s^{-1}f_1\otimes\dots\otimes s^{-1} f_{n}\otimes b_{n+1}\otimes   b_{n+2}\Big) \\
		=&  (-1)^{\gamma}T_n\Big(\kappa(b_1\otimes f_1)\otimes \cdots\otimes  \kappa(b_n\otimes f_n) \Big){\rhd}m_2(b_{n+1}\otimes b_{n+2}) \\
		&-(-1)^{\gamma}T_n\Big(\kappa(b_1\otimes f_1)\otimes \cdots\otimes    \kappa(b_n\otimes sm_2(s^{-1}f_n\otimes  b_{n+1}) )    \Big){\rhd}b_{n+2} \\
		&-(-1)^{\gamma}m_2\Big(   T_n\big(\kappa(b_1\otimes f_1)\otimes \cdots\otimes  \kappa(b_n\otimes f_n) \big){\rhd} b_{n+1} \otimes b_{n+2}\Big)  \\
		=&  (-1)^{\gamma}T_n\Big(\kappa(b_1\otimes f_1)\otimes \cdots\otimes  \kappa(b_n\otimes f_n) \Big){\rhd}(b_{n+1}*b_{n+2}) \\
		&-(-1)^{\gamma}T_n\Big(\kappa(b_1\otimes f_1)\otimes \cdots\otimes    \kappa(b_n\otimes f_n\btl b_{n+1} )    \Big){\rhd}b_{n+2} \\
		&-(-1)^{\gamma}   \Big( T_n\Big(\kappa(b_1\otimes f_1)\otimes \cdots\otimes  \kappa(b_n\otimes f_n) \Big){\rhd} b_{n+1}\Big)* b_{n+2} \\
		=&  (-1)^{\gamma}T_n\Big(\kappa(b_1\otimes f_1)\otimes \cdots\otimes  \kappa(b_n\otimes f_n) \Big){\rhd}(b_{n+1}* b_{n+2}) \\
		&	-(-1)^{\gamma}T_n\Big(\kappa(b_1\otimes f_1)\otimes \cdots\otimes    \kappa(b_n\otimes  f_n    ) {\btl}  b_{n+1}  \Big){\rhd}b_{n+2} \\
		&-(-1)^{\gamma}   \Big( T_n\Big(\kappa(b_1\otimes f_1)\otimes \cdots\otimes  \kappa(b_n\otimes f_n) \Big){\rhd} b_{n+1}\Big)* b_{n+2} .
	\end{align*}
	Thus, the Stasheff identity holding for the element $a_1\otimes s^{-1}f_1\otimes\dots\otimes s^{-1} f_{n}\otimes b_{n+1}\otimes   b_{n+2} $  is equivalent to that $ T_n$  is an $n$-derivation relative to $B$.  
	
	\smallskip
	
		\textbf{Case IV}: for   $b_0,\dots,b_{n+1}\in B$  and $f_1,\dots, f_{n} \in B^\vee$,
	\begin{eqnarray*}	
		&&\sum\limits_{    i+j+k= 2n+2,\atop
			i, k\geqslant  0, j\geqslant  1 } (-1)^{i+jk}m_{i+1+k}\circ\Big(\id^{\ot i}\ot m_j\ot \id^{\ot k}\Big)( b_0\otimes b_1\otimes  s^{-1}f_1\otimes b_2\otimes s^{-1}f_2 \dots b_n\otimes s^{-1} f_{n}\otimes    b_{n+1})\\
		=&&\Big( -m_{2n+1}\big(\id\otimes m_2\otimes \id^{\otimes 2n-1}\big)+m_{2n+1}\big(m_2\otimes \id^{\otimes 2n}\big)-m_2\big(\id\otimes m_{2n+1}\big)\Big)\\
		&&( b_0\otimes b_1\otimes  s^{-1}f_1\otimes b_2\otimes s^{-1}f_2 \otimes\cdots b_n\otimes s^{-1} f_{n}\otimes    b_{n+1}) \\
		=&&-(-1)^{\gamma+n|b_0|}T_n\Big(\kappa(b_0\otimes b_1\lhd s^{-1}f_1)\otimes \kappa( b_2\otimes s^{-1}f_2)\otimes\cdots\otimes\kappa( b_n\otimes s^{-1}f_n)\Big)\rhd b_{n+1}\\
		&&+(-1)^{\gamma+n|b_0|}T_n\Big(\kappa(b_0*b_1\otimes s^{-1}f_1)\otimes \kappa( b_2\otimes s^{-1}f_2)\otimes\cdots\otimes\kappa( b_n\otimes s^{-1}f_n)\Big)\rhd b_{n+1}\\
		&&-(-1)^{\gamma+|b_0|}\Big( b_0\btr T_n\Big(\kappa(b_1\otimes s^{-1}f_1)\otimes \kappa( b_2\otimes s^{-1}f_2)\otimes\cdots\otimes\kappa( b_n\otimes s^{-1}f_n)\Big)\Big) \rhd b_{n+1}
	\end{eqnarray*}
	Thus, the Stasheff identity holding for the element $b_0\otimes b_1\otimes  s^{-1}f_1\otimes b_2\otimes s^{-1}f_2 \otimes\cdots b_n\otimes s^{-1} f_{n}\otimes    b_{n+1} $  is equivalent to that $ T_n$  satisfies Equation~(\ref{Eq:Cyclic Leibniz identity for homotopy RB at the first componet}) .

\smallskip
	
	\textbf{Case V}: for $1<   l\leqslant n$, $b_1,\dots,b_{n+2}\in B$  and $f_1,\dots,f_{n}\in B^\vee$,
	\begin{align*}
		&\sum\limits_{    i+j+k= 2n+2,\atop
			i, k\geqslant  0, j\geqslant  1 }(-1)^{i+j k} m_{i+1+k} \circ\Big(\mathrm{Id}^{\otimes i} \otimes m_j \otimes \mathrm{Id}^{\otimes k}\Big)\Big(b_1 \otimes s^{-1} f_1 \otimes \cdots\otimes s^{-1} f_{l-1} \otimes b_l \otimes b_{l+1} \otimes s^{-1} f_l \otimes \cdots \otimes s^{-1} f_n \otimes b_{n+2}\Big)\\
		=&\Big( -m_{2n+1}\big(\id^{\otimes 2l-3}\otimes m_2\otimes\id^{\otimes 2(n-l)+2}\big)+ m_{2n+1}\big(\id^{\otimes 2l-2}\otimes m_2\otimes\id^{\otimes 2(n-l)+1}\big)-m_{2n+1}\big(\id^{\otimes 2l-1}\otimes m_2\otimes\id^{\otimes 2(n-l)}\big)\Big)\\
		&\Big(b_1 \otimes s^{-1} f_1 \otimes \cdots \otimes b_l \otimes b_{l+1} \otimes s^{-1} f_l \otimes \cdots \otimes s^{-1} f_n \otimes b_{n+2}\Big)\\
			=&-m_{2n+1}\Big(b_1\otimes \cdots\otimes m_2(s^{-1}f_{l-1}\otimes b_l)\otimes b_{l+1}
		\otimes \cdots\otimes s^{-1}f_n\otimes b_{n+2}\Big)\\
		&+m_{2n+1}\Big(b_1\otimes \cdots\otimes s^{-1}f_{l-1}\otimes m_2(b_l\otimes b_{l+1})\otimes 
		s^{-1}f_{l}\otimes \cdots\otimes s^{-1}f_n\otimes b_{n+2}\Big)\\
		&-m_{2n+1}\Big(b_1\otimes \cdots\otimes b_l\otimes m_2(b_{l+1}\otimes s^{-1}f_{l})\otimes b_{l+2}
		\cdots\otimes s^{-1}f_n\otimes b_{n+2}\Big)\\
		=&-(-1)^{\gamma+\sum\limits_{k=l+1}^{n+1}|b_k|}T_n\Big(\kappa(b_1\otimes f_1)\otimes\cdots\otimes\kappa(b_{l-1}\otimes f_{l-1})\btl b_l\otimes\kappa (b_{l+1}\otimes f_{l})\otimes\cdots\otimes\kappa (b_{n+1}\otimes f_n) \Big)\rhd b_{n+2}\\
		&+(-1)^{\gamma+\sum\limits_{k=l+1}^{n+1}|b_k|}T_n\Big(\kappa(b_1\otimes f_1)\otimes\cdots\otimes\kappa(b_{l-1}\otimes f_{l-1})\otimes\kappa (b_l*b_{l+1}\otimes f_{l})\otimes\cdots\otimes\kappa (b_{n+1}\otimes f_n) \Big)\rhd b_{n+2}	 \\
		&-(-1)^{\gamma+\sum\limits_{k=l+1}^{n+1}|b_k|}T_n\Big(\kappa(b_1\otimes f_1)\otimes\cdots\otimes\kappa(b_{l-1}\otimes f_{l-1})\otimes\kappa (b_l \otimes b_{l+1}\btr f_{l})\otimes\cdots\otimes\kappa (b_{n+1}\otimes f_n) \Big)\rhd b_{n+2}.
	\end{align*}	
	So, we can see that for each $1<  l\leqslant n$ the Stasheff identity holding for the element $b_1 \otimes s^{-1} f_1 \otimes \cdots \otimes s^{-1} f_{l-1}\otimes b_l \otimes b_{l+1} \otimes s^{-1} f_l \otimes \cdots \otimes s^{-1} f_n \otimes b_{n+2} $ is equivalent to that $ T_n$ satisfies Equation~\eqref{Eq:Cyclic Leibniz identity for homotopy RB} for $l$.

	In conclusion,     $(\partial_{-1}B, \{m_n\}_{n\geqslant1}) $ is an $A_\infty$ algebra.
	\end{proof}
	\bigskip

\end{document}